\documentclass[11pt,reqno]{amsart}
\usepackage{amsmath, amsfonts, amsthm, amssymb, graphicx, xcolor}

\theoremstyle{plain}
\newtheorem{theorem}{Theorem}[section]
\newtheorem{lemma}{Lemma}[section]
\newtheorem{proposition}{Proposition}[section]
\theoremstyle{definition}

\theoremstyle{remark}
\newtheorem{remark}{Remark}[section]
\newcommand{\rb}{\bar{\rho}}       % Barred \rho
\renewcommand{\xi}{{s}}

\newcommand{\grad}{\nabla}            % Gradient
\newcommand{\e}{\varepsilon}          % Epsilon

\newcommand{\R}{\mathbb{R}}           % 1-d Space
\newcommand{\supp}{\mbox{supp}}
\newcommand{\rhob}{\bar{\rho}}

\newcommand{\etat}{\tilde{\eta}}
\newcommand{\qt}{\tilde{q}}

\newcommand{\cd}{\cdot}
\newcommand{\rhot}{\tilde{\rho}}
\newcommand{\x}{\mathbf{x}}
\newcommand{\vv}{\mathbf{v}}

\renewcommand{\rhot}{\rhob}

\def\be{\begin{equation}}
\def\ee{\end{equation}}
\def\bes{\begin{equation*}}
\def\ees{\end{equation*}}
\def\bc{\begin{cases}}
\def\ec{\end{cases}}

\numberwithin{equation}{section}

\newcommand{\etach}{\check{\eta}}
\newcommand{\qch}{\check{q}}

\author{Gui-Qiang G. Chen}
\address{Gui-Qiang G. Chen, Mathematical Institute, University of Oxford,
Andrew Wiles Building, Radcliffe Observatory Quarter, Woodstock Road,
Oxford OX2 6GG, UK; Academy of Mathematics and Systems Science,
Chinese Academy of Sciences, Beijing 100190, China.}
\email{chengq@maths.ox.ac.uk}

\author{Mikhail Perepelitsa}
\address{Mikhail Perepelitsa, Department of Mathematics, University of Houston,
651 PGH Building, Houston, Texas 77204--3008, USA}
\email{misha@math.uh.edu}

\title[Global Solutions to the Euler Equations with Spherical Symmetry]{Vanishing Viscosity Solutions
of \\ the Compressible Euler Equations with \\ Spherical Symmetry and Large Initial Data}

\begin{document}
\begin{abstract}
We are concerned with spherically symmetric solutions of the Euler equations
for multidimensional compressible fluids,
which are motivated by many important
physical situations.
Various evidences indicate that spherically symmetric solutions of
the compressible Euler equations may blow up near the origin
at certain time under some circumstance.
The central feature
is the strengthening of waves as they move radially inward.
A longstanding open, fundamental question is whether concentration
could form at the origin.
In this paper, we develop a method of vanishing viscosity
and related estimate techniques for viscosity approximate solutions,
and establish the convergence of the approximate solutions to a global
finite-energy entropy solution of the compressible
Euler equations with spherical symmetry and large initial data.
This indicates that concentration does not form in the vanishing
viscosity limit, even though the density may blow up at certain time.
To achieve this,
we first construct global smooth solutions of appropriate
initial-boundary value problems
for the Euler equations with designed viscosity terms,
an approximate pressure function,
and boundary conditions,
and then we establish the strong convergence of the viscosity approximate
solutions to a finite-energy entropy solutions of the Euler equations.
\end{abstract}

\maketitle

\begin{section}{Introduction}

\medskip
We are concerned with the existence theory for spherically symmetric global solutions
of the Euler equations for multidimensional compressible homentropy fluids:
\begin{equation}\label{1.1}
\begin{cases}
\partial_t \rho+\nabla_\x\cdot (\rho \vv)=0,\\[1mm]
(\rho\vv)_t+\nabla_\x\cdot(\rho \vv\otimes \vv)+\nabla_\x p=0,\\
\end{cases}
\end{equation}
where $\rho\geq 0$ is the density, $p$ the pressure,  $\vv\in \mathbb{R}^n$ the
velocity, $t\in\R$, $\x\in\R^n$,
and $\nabla_\x$ is the gradient with respect to $\x\in\R^n$.
The constitutive pressure-density relation for polytropic perfect gases is
$$
p=p(\rho)=\kappa\rho^\gamma,
$$
where $\gamma> 1$ is the adiabatic exponent and, by scaling, the constant $\kappa$ in the
pressure-density relation may be chosen as $\kappa= (\gamma-1)^2/4\gamma$ without loss of generality.

For the spherically symmetric motion,
\begin{equation}\label{1.2a}
\rho(t,\x)=\rho(t,r), \quad
\vv(t,\x)=u(t,r)\frac{\x}{r}, \qquad\,\;\; r=|\x|.
\end{equation}
Then the functions $(\rho, m)=(\rho, \rho u)$ are governed by the following Euler equations with
geometrical terms:
\begin{equation}
\label{eq:Eu}
\left\{ \begin{array}{l}
\partial_t\rho{}+\partial_r m{}+{}\frac{n-1}{r}m{}={}0,\\[2mm]
\partial_t m{}+{}\partial_r (\frac{m^2}{\rho}{}+{}p(\rho)){}+{}\frac{n-1}{r}\frac{m^2}{\rho}{}={}0.
\end{array}
\right.
\end{equation}
The existence theory for spherically symmetric solutions $(\rho, \vv)(t,\x)$
to \eqref{1.1} through form \eqref{1.2a}
is equivalent to the existence theory for global solutions $(\rho, m)(t,r)$
to \eqref{eq:Eu}.
For any problem with a constant velocity $\vv_\infty$ at infinity, {\it i.e.},
$\lim_{|\x|\to \infty} \vv(t, \x)=\vv_\infty$,
we may assume without loss of generality that $\vv_\infty=0$,
or equivalently $\lim_{r\to \infty} u(t, r)=0$, by the Galilean invariance.

The study of spherically symmetric solutions can date back 1950s,
which are motivated by many important
physical problems such as flow in a jet engine inlet manifold and
stellar dynamics including gaseous stars and supernovae formation.
In particular, the similarity solutions of such a problem have been
discussed in
a large literature ({\it cf.} \cite{CFr,Guderley,Ro,Slemrod,Wh}),
which are determined by singular ordinary differential equations.
The central feature is the strengthening of waves as they move radially inward.
Various evidences indicate that spherically symmetric solutions of
the compressible Euler equations may blow up near the origin at certain time under
some circumstance.
A longstanding open, fundamental question is whether concentration
could form at the origin, that is, the density becomes a delta measure at the origin,
especially when a focusing spherical shock is moving inward
the origin ({\it cf.} \cite{CFr,Ro,Wh}).

Some progress has been made for solving this problem
in the recent decades.
The local existence of spherically symmetric weak solutions outside a solid ball
at the origin was discussed in Makino-Takeno \cite{MT} for the case $1 <\gamma\le \frac{5}{3}$;
also see Yang \cite{Yang}.
A shock capturing scheme was introduced in Chen-Glimm \cite{ChenGlimm} for
constructing approximate solutions to spherically symmetric entropy solutions
for $\gamma>1$,
where the convergence proof was limited to be locally in time.
A first global existence of entropy solutions including the origin
was established in Chen \cite{Chen} for a class of $L^\infty$ Cauchy data
of arbitrarily large amplitude,
which model outgoing blast waves and large-time asymptotic solutions.
Also see
Slemrod
\cite{Slemrod} for the resolution of the spherical
piston problem for compressible homentropic gas dynamics
via a self-similar viscous limit
and
LeFloch-Westdickenberg
\cite{LW} for a compactness framework
to ensure the strong compactness of spherically symmetric approximate
solutions with uniform finite-energy norms for the case $1<\gamma\le \frac{5}{3}$.

\smallskip
The approach and ideas developed in this paper
yield indeed the global existence of finite-energy entropy solutions
of the compressible
Euler equations with spherical symmetry and large initial data,
for the general case $\gamma>1$,
based on our earlier results in \cite{CP}.
To establish the existence of global entropy solutions to \eqref{eq:Eu} with initial data:
\begin{equation}\label{ID}
(\rho, m)|_{t=0}=(\rho_0, m_0),
\end{equation}
we develop a method of vanishing viscosity and related estimate techniques for viscosity approximate solutions,
and establish the convergence of the viscosity approximate solutions to a global finite-energy entropy solution.
To achieve this, we first construct global smooth solutions of appropriate initial-boundary value problems
for the Euler equations with designed viscosity terms,
an approximate pressure function,
and boundary conditions,
and then we establish the strong convergence of the viscosity approximate solutions to an entropy solution of the Euler
equations \eqref{eq:Eu}, which is equivalent to \eqref{1.1} via relation \eqref{1.2a}.
For simplicity of presentation, we focus our analysis on the physical region $1<\gamma\le 3$ throughout the paper,
though the convergence argument also works for all $\gamma>1$.

The viscosity terms and approximate pressure function are designed to approximate the Euler equations are as follows:

\begin{equation}
\label{eq:NS}
\left\{ \begin{array}{l}
\rho_t{}+{}m_r{}+{}\frac{n-1}{r}m{}={}\e\big(\rho_{rr}{}+{}\frac{n-1}{r}\rho_r\big)\equiv \e r^{-(n-1)}\big(r^{n-1}\rho_r\big)_r,\\[2mm]
m_t{}+{}(\frac{m^2}{\rho}{}+{}p_\delta(\rho))_r{}+{}\frac{n-1}{r}\frac{m^2}{\rho}
 {}={}\e\big(m_r{}+{}\frac{n-1}{r}m\big)_r\equiv \e \big(r^{-(n-1)}(r^{n-1}m)_r\big)_r,
\end{array}
\right.
\end{equation}
where
$$
p_\delta(\rho){}=\kappa\rho^\gamma +\delta \rho^2, \qquad \delta=\delta(\e)>0,
$$
with $\e\in (0,1]$ and $\delta(\e)\to 0$ as $\e\to 0$ in an appropriate order.
Notice that the positive term $\delta\rho^2$ is added into $p_\delta(\rho)$
to avoid the possibility of formation of cavitation of the solutions to
the viscous system \eqref{eq:NS}.

We consider \eqref{eq:NS}
on cylinder $Q^\e{}={}\R_+\times (a, b)$, with $\R_+=[0, \infty)$,
$a:=a(\e)\in (0, 1), b:=b(\e)>1$,
and
$$
\lim_{\e\to 0} a(\e){}={}0, \qquad \lim_{\e\to 0} b(\e)=\infty,
$$
with the boundary conditions:
\begin{equation}
 \label{eq:bc}
(\rho_r, m)\big|_{r=a}{}=(0,0), \quad (\rho, m)|_{r=b}=(\rhob, 0)
\qquad\;\; \mbox{for}\,\,t>0,
\end{equation}
for some $\rhob:=\rhob(\e)>0$, and with appropriate approximate initial functions:
\begin{equation}
\label{eq:ic}
(\rho, m)|_{t=0}{}={}(\rho_0^\e, m_0^\e)(r)  \qquad\mbox{for}\;\; a<r<b,
\end{equation}
satisfying the conditions in Theorem \ref{main} below.

A pair of mappings $(\eta, q): \R_+ \times\R\to \R^2$ is
called an entropy-entropy flux pair (or entropy pair, for short) of
system \eqref{eq:Eu} if the pair satisfies the $2\times 2$
linear hyperbolic system:
\begin{equation}
\nabla q(U)=\nabla \eta(U)\,
\nabla
\left(\begin{array}{c}
 m\\
 \frac{m^2}{\rho}+p(\rho)
\end{array}
\right),
\end{equation}
where $\nabla=(\partial_\rho, \partial_m)$
is the gradient with respect to $U=(\rho, m)$ from now on.
Furthermore, $\eta(\rho,m)$ is called a weak entropy if
\begin{equation}
\eta\Big|_{\begin{subarray}{l}
                \rho=0\\
                u=m/\rho \,\,\text{fixed}
        \end{subarray}}=0.
\label{7.1.1.2}
\end{equation}
An entropy pair is said to be convex if the Hessian $\nabla^2\eta(\rho,m)$
is nonnegative
in the region under consideration.

For example, the mechanical energy $\eta^*(\rho, m)$ (a sum of the kinetic and
internal energy) and the mechanical energy flux $q^*(\rho, m)$:
\begin{equation}
\eta^*(\rho,m)=\frac{1}{2}\frac{m^2}{\rho}+ \frac{\kappa \rho^\gamma}{\gamma-1}, \quad
q^*(\rho,m)=\frac{1}{2}\frac{m^3}{\rho^2}+ \frac{\kappa\gamma}{\gamma-1}m\rho^{\gamma-1},
\label{mech-energy}
\end{equation}
form a special entropy pair of system \eqref{eq:Eu}; $\eta^*(\rho,m)$ is convex
in the region $\rho\ge 0$.

\medskip
Any weak entropy pair for the Euler system \eqref{eq:Eu} can be
expressed by
\begin{eqnarray}
\label{eta}
\eta_\psi(\rho,m){}&=&{}\rho\int_{-\infty}^\infty \psi(\frac{m}{\rho}+\rho^\theta s)[1-s^2]^{\lambda}_+\,ds, \\
\label{q}
q_\psi(\rho,m){}&=&{}\rho \int_{-\infty}^\infty(\frac{m}{\rho}+\theta \rho^\theta s) \psi(\frac{m}{\rho}+\rho^\theta s)[1-s^2]^{\lambda}_+\,ds,
\end{eqnarray}
with $\lambda{}={} \frac{3-\gamma}{2(\gamma-1)}$ and the generating function $\psi(s)$.

\begin{theorem}
\label{main}
Assume that $(\rho_0,m_0)\in L^1_{loc}(\R_+)^2$, with $\rho_0\geq 0$, is of finite energy:
\begin{equation}\label{ID-a}
\big( \frac{m_0^2}{2\rho_0}{}+{}\frac{\kappa\rho_0^\gamma}{\gamma-1}\big)r^{n-1}\in L^1(\R_+).
\end{equation}
Let $(\delta,\rhob){}={}(\delta(\e),\,\rhob(\e))\in (0,\e)\times (0,1)$
with $\lim_{\e\to 0}(\delta,\, \rhob){}={}(0,\,0)$
satisfy
\begin{equation}\label{epsilon-delta}
\rhob^\gamma b^n+\frac{\delta}{\e}b^n\le M,
\end{equation}
for some $M<\infty$ independent of $\e\in (0,1]$.
If $(\rho_0^\e,m_0^\e)$ is a sequence of smooth functions with the
following properties:
\begin{itemize}
\item[(i)] $\rho_0^\e>0$;

\medskip
\item[(ii)] $(\rho^\e_0,m^\e_0)$ satisfies \eqref{eq:bc} and
\begin{equation}\label{1.6a}
\big(r^{n-1} m^\e_{0}\big)_r\big|_{r=a}{}={}0,
\end{equation}
and, at $r=b$,
\begin{equation}\label{1.6b}
\qquad m_{0,r}^\e {}={} \e r^{-(n-1)}\big(r^{n-1}\rho_{0,r}^\e\big)_r, \quad
\Big(\frac{(m_0^\e)^2}{\rho_0^\e}{}+{}p_\delta(\rho_0^\e)\Big)_r
{}={}\e r^{-(n-1)}\big(r^{n-1}m_{0}^\e\big)_r;
\end{equation}

\smallskip
\item[(iii)] $(\rho_0^\e,m_0^\e){}\to{}(\rho_0,m_0)$ a.e. $r\in \R_+$ as $\e\to 0$,
where we understand $(\rho_0^\e,m_0^\e)$ as the zero extension of $(\rho_0^\e,m_0^\e)$ outside $(a, b)$;

\medskip
\item[(iv)] $\int_a^b \Big(\frac{(m_0^\e)^2}{2\rho^\e_0}{}+{}\frac{\kappa (\rho_0^\e)^\gamma}{\gamma-1}\Big)r^{n-1}dr
\to \int_0^\infty \Big(\frac{m_0^2}{2\rho_0}+\frac{\kappa \rho_0^\gamma}{\gamma-1}\Big)r^{n-1}dr$
as $\e\to 0$,
\end{itemize}

\medskip
\noindent
then, for each fixed $\e>0$, there is a unique global classical solution $(\rho^\e,m^\e)(t,r)$
of \eqref{eq:NS}--\eqref{eq:ic} with initial data $(\rho_0^\e,m_0^\e)$
so that there exists a subsequence (still labeled $(\rho^\e,m^\e)$) that converges
a.e. $(t,r)\in \R_+^2:=\R_+\times \R_+$
and  in $L^p_{loc}(\R_+^2)\times L^{q}_{loc}(\R_+^2)$, $p\in[1,\gamma+1)$, $q\in[1, \frac{3(\gamma+1)}{\gamma+3})$,
as $\e\to 0$, to a global finite-energy entropy solution $(\rho,m)$ of the Euler equations \eqref{eq:Eu} with initial condition
\eqref{eq:ic} in the following sense:
\begin{enumerate}
\item[(i)] For any $\varphi{}\in{}C^\infty_0(\R_+^2)$ with $\varphi_r(t,0){}=0$,
\[
\int_{\R_+^2}\big(\rho\varphi_t{}+{}m\varphi_r\big)\,r^{n-1}drdt{}
+{}\int_0^\infty \rho_0(r)\varphi(0,r)\,r^{n-1}dr{}={}0;
\]
\item[(ii)] For all $\varphi{}\in{}C^\infty_0(\R_+^2)$, with $\varphi(t,0){}={}\varphi_r(t,0){}={}0$,
\[
\int_{\R_+^2} \big(m\varphi_t{}+ \frac{m^2}{\rho}\varphi_r +p(\rho) (\varphi_r+\frac{n-1}{r}\varphi)\big)\,r^{n-1}drdt{}+{}
\int_0^\infty m_0(r)\varphi(0,r)\,r^{n-1}dr{}={}0;
\]
\item[(iii)] For a.e. $t_2\ge t_1\ge 0$,
\begin{equation}
\label{finite_energy}
\int_0^\infty \eta^*(\rho, m)(t_2,r)\,r^{n-1}dr{}
\leq{}\int_0^\infty \eta^*(\rho, m)(t_1,r)\,r^{n-1}dr{}
{}\leq{} \int_0^\infty \eta^*(\rho_0, m_0)(r)\,r^{n-1}dr;
\end{equation}
\item[(iv)]
For any convex function $\psi(s)$ with subquadratic growth at infinity and any entropy pair $(\eta_\psi,q_\psi)$
defined in \eqref{eta}--\eqref{q},
\begin{equation}
\label{entropy_sol}
(\eta_{\psi}r^{n-1})_t{}+{}(q_\psi r^{n-1})_r {}
+{}(n-1)r^{n-2}\big(m\eta_{\psi,\rho}{}+{}\frac{m^2}{\rho}\eta_{\psi,m}{}-{}q_\psi\big){}\leq{}0
\end{equation}
in the sense of distributions.
\end{enumerate}
\end{theorem}

\begin{remark}
Theorem \ref{main} indicates that there is no concentration formed in the vanishing viscosity limit of the viscosity
approximate solutions to the global entropy solution of the compressible Euler equations \eqref{eq:Eu}
with initial condition \eqref{eq:ic}, which
is of finite-energy \eqref{finite_energy} and obeys the entropy inequality \eqref{entropy_sol}.
\end{remark}

\begin{remark}
To achieve \eqref{epsilon-delta}, it suffices to choose
$\delta{}={}\e b^{-{k_1}}$ and $\rhob{}={}b^{-k_2}$
for any $k_1\ge n$ and $k_2\ge \frac{n}{\gamma}$.
\end{remark}
\end{section}

\begin{section}{Global Existence of a Unique Classical
Solution of the Approximate Euler Equations with Artificial Viscosity}

The equations in \eqref{eq:NS} form a quasilinear parabolic system for $(\rho, m)$.
In this section, we show the existence of a unique smooth solution $(\rho, m)$,
equivalently $(\rho, u)$ with $u=\frac{m}{\rho}$,
and make some estimates of the solution whose bounds may depend on the parameter $\e\in (0,1]$
(except the energy bound $E_0$ below).
For $\beta\in(0, 1)$, let $C^{2+\beta}([a,b])$ and $C^{2+\beta,1+\frac{\beta}{2}}(Q_T)$
be the usual H\"older and parabolic H\"older spaces, where
$Q_T=[0,T]\times (a,b)$ ({\it cf.} \cite{Ladyzhenskaja}).
For simplicity, we will drop the $\e$--dependence of the involved functions
in this section.

\begin{theorem}\label{theorem:exist}
Let $(\rho_0,m_0)\in  (C^{2+\beta}([a,b]))^2$ with $\inf_{a\le r\le b}\rho_0(r)>0$
and satisfy \eqref{eq:bc} and \eqref{1.6a}--\eqref{1.6b}.
Then there exists a unique global solution $(\rho, m)$
of problem \eqref{eq:NS}--\eqref{eq:ic} for $\gamma \in (1,3]$ such that
$$
(\rho,m){}\in{} (C^{2+\beta,1+\frac{\beta}{2}}(Q_T))^2, \quad \inf_{Q_T}\rho>0\qquad\,\,\,\,\mbox{for all}\,\,\, T>0.
$$
\end{theorem}

The nonlinear terms in \eqref{eq:NS} have singularities when $\rho=0$ or $|m|=\infty$.
To establish Theorem \ref{theorem:exist}, we derive {\it a priori} estimates for a generic solution in $C^{2,1}(Q_T)$
with $\|(\rho,\frac{1}{\rho},\frac{m}{\rho})\|_{L^\infty(Q_T)}<\infty$,
showing by this that the solution takes values in a region (determined {\it a priori}) away from
the singularities.
With the {\it a priori} estimates, the existence of the solution
can be derived
from the general theory of the quasilinear parabolic systems,
by a suitable linearization techniques;
see
Section 5
and Theorem 7.1 in Ladyzhenskaja-Solonnikov-Uraltseva \cite{Ladyzhenskaja}.

The {\it a priori} estimates are obtained by the following arguments:
First we derive the estimates based on the balance of total energy.
Then, in Lemma \ref{lemma:uniform}, we use the maximum principle for the Riemann invariants
and the total energy estimates to show that the $L^\infty$--norm of $u=\frac{m}{\rho}$ depends linearly
on the $L^\infty$--norm of $\rho^{(\gamma-1)/2}$.
This is in turn used in Lemma \ref{lemma:high-derivative-1}
to close the higher energy estimates for $(\rho_r, m_r)$.
With that, we obtain the {\it a priori} upper bound $\rho$ in $L^\infty$
and, by using Lemma \ref{lemma:uniform} again, the {\it a priori} bounds of the $L^\infty$--norms
of $m$ and $u$.
Finally, to show the positive lower bound for $\rho$,
we obtain an estimate on $\int_0^t\|u_r(t,\cdot)\|_{L^\infty}\,dt$.

\medskip
We proceed now with the derivation of the {\it a priori} estimates.
Let $(\rho, m)$, with $\rho>0$, be a $C^{2,1}(Q_T)$
solution of \eqref{eq:NS}--\eqref{eq:ic}
with
\eqref{1.6a}--\eqref{1.6b}.

\subsection{Energy Estimate}
As usual, we denote by
\begin{equation}\label{2.1a}
 \eta^*_\delta{}={}\frac{m^2}{2\rho}{}+{}h_\delta(\rho),\qquad
 q^*_\delta{}={}\frac{m^3}{2\rho^2}{}+{}m h'_\delta(\rho),
\end{equation}
as the mechanical energy pair of system \eqref{eq:NS} with $\e=0$,
where $h_\delta(\rho):=\rho e_\delta(\rho)$ for the internal energy $e_\delta(\rho):=\int_0^\rho\frac{p_\delta(s)}{s^2}\,ds$.

Note that $(\rhob,0)$ is the only constant equilibrium state of the system.
For the mechanical energy pair $(\eta^*_\delta, q^*_\delta)$ in \eqref{2.1a},
we denote
\begin{equation}
\label{E}
 \bar{\eta}^*_\delta(\rho,m)
 {}={}\eta^*_\delta(\rho,m){}-{}\eta^*_\delta(\rhob,0){}-{}(\eta^*_\delta)_\rho(\rhob,0)(\rho-\rhob),
\end{equation}
as the total energy relative to the constant equilibrium state $(\rhob,0)$.

\begin{proposition}\label{2.1}
Let
$$
E_0:=\sup_{\e>0}\int_a^b \bar{\eta}^*_\delta(\rho_0^\e(r), m_0^\e(r))r^{n-1}dr <\infty.
$$
Then, for the viscosity approximate solution $(\rho, m)=(\rho, \rho u)$ determined by Theorem {\rm \ref{theorem:exist}}
for each fixed $\e>0$,
we have
\begin{eqnarray}
&&
\sup_{t\in[0,T]}\int_a^b \big(\frac{1}{2}\rho u^2
+ \bar{h}_\delta(\rho, \rhob)\big) r^{n-1} dr
\notag\\
&&\qquad
+{}\e\int_{Q_T}\Big(  h''_\delta(\rho)|\rho_r|^2{}+{}\rho|u_r|^2+(n-1)\frac{\rho u^2}{r^2}\Big)\,r^{n-1}drdt{}
{}\le {}E_0,
\label{Energy_est}
\end{eqnarray}
where
\begin{equation}\label{Energy_est-a}
\bar{h}_\delta(\rho, \rhob)=h_\delta(\rho)-h_\delta(\rhob)-h'_\delta(\rhob)(\rho-\rhob)
\ge c_1\rho(\rho^\theta-\rhob^\theta)^2, \qquad \theta=\frac{\gamma-1}{2},
\end{equation}
for some constant $c_1=c_1(\rhob,\gamma)>0$.
Furthermore, for any $t\in[0,T]$,
the measure of set $\{\rho(t,\cdot)>\frac{3}{2}\rhob\}$ is less than $c_2E_0$
for some $c_2=c_2(\rhob,\gamma)>0$.
\end{proposition}

\begin{proof}
We multiply the first equation in \eqref{eq:NS} by $(\bar{\eta}^*_\delta)_\rho r^{n-1}$,
the second in \eqref{eq:NS} by $(\bar{\eta}^*_\delta)_m r^{n-1}$,
and then add them up to obtain
\begin{eqnarray*}
\label{eq:energy-1}
&&\big(\bar{\eta}^*_\delta r^{n-1}\big)_t
{}+{}\big((q^*_\delta-(\eta^*_\delta)_\rho(\rhob,0)m)r^{n-1}\big)_r{}\notag\\
&&= \e r^{n-1}\big( \rho_{rr}{}+{}\frac{n-1}{r}\rho_r\big)\big((\eta^*_\delta)_\rho-(\eta^*_\delta)_\rho(\rhob,0)\big)
{}+{}\e r^{n-1}\big(m_r+\frac{n-1}{r}m\big)_r(\eta^*_\delta)_m,
\end{eqnarray*}
that is,
\begin{eqnarray}
\label{eq:energy-2}
&&(\bar{\eta}^*_\delta r^{n-1})_t{}+{}\big((q^*_\delta-(\eta^*_\delta)_\rho(\rhob,0)m)r^{n-1}\big)_r{}
+{} (n-1)\e m (\eta^*_\delta)_mr^{n-3}\notag\\[2mm]
&&= \e ( \rho_{r}r^{n-1})_r\big((\eta^*_\delta)_\rho-(\eta^*_\delta)_\rho(\rhob,0)\big)
  +\e (m_rr^{n-1})_r (\eta^*_\delta)_m.
\end{eqnarray}
Integrating both sides of \eqref{eq:energy-2} over $Q_t$ for any $t\in (0,T]$ and
using the boundary conditions \eqref{eq:bc}, we have
\begin{eqnarray*}
&&
\int_a^b \bar{\eta}^*_\delta r^{n-1}\,dr
{}+{}\e\int_{Q_t}\Big(
 (\rho_r,m_r)\grad^2 \bar{\eta}^*_\delta(\rho_r,m_r)^\top{}+{}\frac{m^2}{2\rho r^2}\Big)
r^{n-1}\,drdt{}={}E_0. \qquad \label{Energy_est-b}
\end{eqnarray*}
Note that $(\rho_r,m_r)\grad^2\bar{\eta}^*_\delta(\rho_r,m_r)^\top$ is a positive
quadratic form that dominates $h''_\delta(\rho)|\rho_r|^2$ and $\rho|u_r|^2$ so that
\begin{equation}
\label{est:energy_2}
\int_a^b \bar{\eta}^*_\delta r^{n-1}\,dr
{}+{}
\e \int_{Q_T} \left( (2\delta{}+{}\kappa\gamma\rho^{\gamma-2})|\rho_r|^2{}+{}\rho|u_r|^2
+(n-1)\frac{\rho u^2}{r^2}\right)\,r^{n-1}drdt{}\leq{}E_0.
\end{equation}
Estimate \eqref{est:energy_2} also implies
\begin{equation*}
\sup_{t\in[0,T]}\int_a^b\big(\rho u^2
+ \bar{h}_\delta(\rho,\rhob)\big)r^{n-1}\,dr{}\leq{} E_0.
\end{equation*}
The function
$\bar{h}_\delta(\rho,\rhob)$
is positive, quadratic in $\rho-\rhob$ for $\rho$ near $\rhob$,
and grows as $\rho^{\max\{\gamma, 2\}}$
for large values of $\rho$.
In particular, there exists $c_1=c_1(\rhob,\gamma)>0$ such that \eqref{Energy_est-a} holds.
Thus, for any $t\in[0,T]$,
the measure of set $\{\rho(t,\cdot)>\frac{3}{2}\rhob\}$ is less than $c_2 E_0$
for some $c_2>0$.
\end{proof}

\medskip
With the basic energy estimate \eqref{Energy_est}, we have
\begin{lemma}\label{2.1-b}
There exists $C=C(\e, T, E_0)>0$ such that
\begin{equation} \label{est:rho_int_2_gamma}
\int_0^T \|\rho(t,\cdot)\|_{L^\infty(a,b)}^{2\max\{2,\gamma\}}\,dt{}\leq{} C.
\end{equation}
\end{lemma}

\begin{proof}
In the case that the measure of set $\{\rho(t,\cdot)>\frac{3}{2}\rhob\}$ is zero, we have
the uniform upper bound $\frac{3}{2}\rhob$ for $\rho(t,r)$.
Otherwise, for $r\in (a,b)$, let $r_0\in (a,b)$ be the closest to point $r$ such that
$\rho(t,r_0)=\frac{3}{2}\rhob$.
Clearly, $|r-r_0|\leq c(\rhob)E_0$.
With such a choice of $r_0$, we have
\begin{eqnarray}
\label{est:rho_bounded}
&&|\rho^\gamma(t,r)-\rho^\gamma(t,r_0)|\\
&&\leq \gamma\Big|\int_{r_0}^r\rho^{\gamma-1}(t,y) \rho_y(t,y)\,dy\Big| \notag \\
&&\leq C\,\Big|\int_{r_0}^r \rho^\gamma(t,y)y^{n-1}\,dy \Big|^{\frac{1}{2}}
\Big(\int_a^b\rho^{\gamma-2}(t,y)|\rho_y(t,y)|^2y^{n-1}\,dy\Big)^{\frac{1}{2}} \notag\\
\label{est:rho_gamma_energy}
&&\leq{} C\Big(\int_a^b\rho^{\gamma-2}(t,r)|\rho_r(t,r)|^2r^{n-1}\,dr\Big)^{\frac{1}{2}}.
\end{eqnarray}
Then estimate \eqref{est:energy_2} yields
\begin{equation}
\int_0^T \|\rho(t,\cdot)\|_{L^\infty(a,b)}^{2\gamma}\,dt{}\leq{} C,
\end{equation}
where $C$ stands for a generic function of
the parameters: $\gamma,\e,\delta, T, E_0$, and $\bar{\rho}$.

Repeating the argument with $\rho^2$ instead of $\rho^\gamma$, we conclude
\eqref{est:rho_int_2_gamma}.
\end{proof}

From now on, the constant $C>0$ is a universal constant that may {\it depend on the parameter $\e>0$} in \S 2.2--\S 2.3,
while the constant $M>0$ below is another universal constant {\it independent of the parameter $\e$} as $E_0$ from \S 3, though both of
them may also depend on $T>0$, $E_0$, and other parameters; we will also specify their dependence whenever needed.

\begin{subsection}{Maximum Principle Estimates}
Furthermore, we have
\begin{lemma}
\label{lemma:uniform}
There exists $C=C(a,T,E_0)$ such that, for any $t\in[0,T],$
\begin{equation}
 \label{est:u_max}
\|u\|_{L^\infty(Q_t)}\leq{} C\big(\|u_0+R(\rho_0)\|_{L^\infty(a,b)}{}
+{} \|u_0-R(\rho_0)\|_{L^\infty(a,b)}{}+{}\|R(\rho)\|_{L^\infty(Q_t)}\big),
\end{equation}
where
\begin{equation}
\label{Riemann.invariants}
R(\rho){}={}\int_0^\rho\frac{\sqrt{p_\delta'(s)}}{s}\,ds.
\end{equation}
\end{lemma}

\begin{proof}
Consider system \eqref{eq:NS}.
The characteristic speeds of system \eqref{eq:NS} without artificial viscosity terms are
\[
\lambda_1{}={}u-\sqrt{p_\delta'(\rho)},\qquad \lambda_2{}={}u+\sqrt{p_\delta'(\rho)},
\]
and the corresponding right-eigenvectors are
\[
r_1{}={}\left[\begin{array}{c}
1\\\lambda_1 \end{array}\right],
\qquad r_2{}={}\left[\begin{array}{c}
1\\\lambda_2 \end{array}\right].
\]
The Riemann invariants $(w,z)$, defined by the conditions $\grad w\cdot r_1{}={}0$ and $\grad z\cdot r_2{}={}0$,
are given by
\[
w{}={}\frac{m}{\rho}{}+{}R(\rho),\qquad z{}={}\frac{m}{\rho}-R(\rho),
\]
with $R$ defined in \eqref{Riemann.invariants}.
They are quasi-convex:
\begin{equation}
\label{quasi}
\grad^\perp w  \grad^2 w (\grad^\perp w)^\top \geq0,\qquad
-\grad^\perp z \grad^2 z(\grad^\perp z)^\top\geq 0,
\end{equation}
where $\grad^2$ is the Hessian with respect to $(\rho,m)$ and $\grad^\perp{}={}(\partial_m,-\partial_\rho).$

Let us multiply the first equation in \eqref{eq:NS} by $w_{\rho}(\rho,m)$,
the second in \eqref{eq:NS}
by $w_{m}(\rho,m)$, and add them to obtain
\begin{eqnarray*}
&&w_{t}{}+{}\lambda_2 w_{r}{}+{}\frac{n-1}{r} u \sqrt{p_\delta'(\rho)}\\
&&={}-\e\big(\rho_r (w_{\rho})_r{}+{}m_r (w_m)_r\big){}+{}\e w_{rr}
+{} \frac{(n-1)\e}{r}\big(w_{r}{}-{}\frac{1}{r}mw_{m} \big),
\end{eqnarray*}
where $\lambda_2$ is as above.
Then
\begin{eqnarray*}
&&w_{t}{}+{}\big(\lambda_2-\frac{(n-1)\e}{r}\big) w_{r}{}-{}\e w_{rr}\\
&&={}-\e  (\rho_r,m_r)\grad^2 w(\rho_r,m_r)^\top
{}-{} \frac{n-1}{r}u\sqrt{p_\delta'(\rho)}-{}(n-1)\e\frac{u}{r^2}.
\end{eqnarray*}
We write
$$
(\rho_r, m_r){}={}\alpha\grad w{}+{}\beta\grad^\perp w,
$$
with
\[
 \alpha{}={}\frac{w_{r}}{|\grad w|^2},\qquad
 \beta{}={}\frac{\rho_r w_{m}-m_r w_{\rho}}{|\grad w|^2}.
\]
Then we can further write
\begin{eqnarray}\label{RI}
&&w_{t}{}+{}\lambda w_{r}{}-{}\e w_{rr}\notag\\
&&={}-\e \beta^2   \grad^\perp w \grad^2 w(\grad^\perp w)^\top
{}-{} \frac{n-1}{r}u \sqrt{p_\delta'(\rho)}-{}(n-1)\e\frac{u}{r^2},
\end{eqnarray}
where
\begin{eqnarray*}
\lambda{}={}\lambda_2{}-{}\frac{(n-1)\e}{r}
{}+{}\frac{\e\alpha}{|\grad w|^2}    \grad w\grad^2w (\grad w)^\top
{}+{}\frac{2\e \beta  }{|\grad w|^2}  \grad^\perp w \grad^2w(\grad w)^\top .
\end{eqnarray*}

By setting
\[
 \tilde{w}(t,r)
 {}={}w(t,r){}-{}(n-1)\int_0^t\Big\|\frac{\sqrt{p_\delta'(\rho(\tau,r))}u(\tau,r)}{r}
 {}+{}\frac{\e u(\tau,r)}{r^2} \Big\|_{L^\infty(a,b)}\,d\tau,
\]
and using the quasi-convexity property \eqref{quasi} and the classical maximum principle
applied to the parabolic equation \eqref{RI}, we obtain
\[
 \max_{Q_t}\tilde{w}
 {}\leq{} \max\big\{ \max_{(a,b)}w_{0}{},{}\max_{[0,t]\times (\{a\}\cup\{b\})}\tilde{w}\big\},
\]
or
\begin{multline*}
 \max_{Q_t} w{}\leq{}\max_{(a,b)} w_{0}{}+{}\|R(\rho)\|_{L^\infty(Q_t)}
 {}+{}C(\rhob,a)\int_0^t\Big(1+\|\rho(\tau,\cdot)\|^{\frac{1}{2}\max\{1,\gamma-1\}}_{L^\infty(a,b)}\Big)\|u(\tau,\cdot)\|_{L^\infty(a,b)}\,d\tau.
\end{multline*}
Similarly, we have
\begin{multline*}
 \max_{Q_t} (-z{})\leq{}\max_{(a,b)} (-z_{0}){}+{}\|R(\rho)\|_{L^\infty(Q_t)}
 {}+{}C\int_0^t\Big(1+\|\rho(\tau,\cdot)\|^{\frac{1}{2}\max\{1, \gamma-1\}}_{L^\infty(a,b)}\Big)\|u(\tau,\cdot)\|_{L^\infty(a,b)}\,d\tau.
\end{multline*}
Since $\rho\geq0$, it follows that
\begin{eqnarray}
\max_{Q_t} |u|&\leq& \max_{(a,b)} |w_{0}|{}+{}\max_{(a,b)} |z_{0}|{}+{}\|R(\rho)\|_{L^\infty(Q_t)}\notag \\
&& {}+{}C(a) \int_0^t\Big(1
 +\|\rho(\tau,\cdot)\|^{\frac{1}{2}\max\{1,\gamma-1\}}_{L^\infty(a,b)}\Big)\|u(\tau,\cdot)\|_{L^\infty(a,b)}
 \,d\tau.\label{est:u_max_0.1}
\end{eqnarray}
By \eqref{est:rho_int_2_gamma} and $\max\{1, \gamma-1\}<4\gamma$, we have
$$
\int_0^T \|\rho(\tau,\cdot)\|^{\frac{1}{2}\max\{1, \gamma-1\}}_{L^\infty}\,d\tau
{}\leq{}C.
$$
Then we conclude \eqref{est:u_max} from \eqref{est:u_max_0.1}.
\end{proof}
\end{subsection}

\begin{subsection}{Lower Bound on $\rho$}
\begin{lemma}\label{lemma:high-derivative-1}
There exists $C{}={}C ( \|(\rho_0,u_0)\|_{L^\infty(a,b)}, \|(\rho_0,m_0)\|_{H^1(a,b)}, \gamma)$
such that
\begin{equation}
 \sup_{t\in[0,T]}\int_a^b \big(|\rho_r|^2{}+{}|m_r|^2\big)\,dr
{}+{}\int_{Q_T} \big(|\rho_{rr}|^2
{}+{}|m_{rr}|^2)\,drdt {}\leq{} C.
\label{2.10a}
\end{equation}
\end{lemma}

\begin{proof} We multiply the first equation in \eqref{eq:NS} by $\rho_{rr}$ and
the second by $m_{rr}$ to obtain
\begin{eqnarray*}
&&-\partial_t\Big(\frac{|\rho_r|^2+|m_r|^2}{2}\Big)
{}-{}\e\big(|\rho_{rr}|^2{}+{}|m_{rr}|^2\big){}+{}
(\rho_t\rho_r)_r{}+{}(m_tm_r)_r \notag \\
&&=-m_r\rho_{rr}{}-{}\frac{(n-1)}{r}m\rho_{rr}
{}-{}(\rho u^2{}+{}p_\delta)_rm_{rr}{}-{}\frac{n-1}{r}\rho u^2 m_{rr}
  \notag \\
&&\quad+\frac{(n-1)\e}{r}\rho_r\rho_{rr}{}+{}\big(\frac{(n-1)\e}{r} m\big)_r m_{rr}.
 \end{eqnarray*}
We integrate this over $Q_t$ to obtain
\begin{eqnarray}
&&\int_a^b \Big(\frac{|\rho_r|^2+|m_r|^2}{2}\Big)\Big|^t_0\,dr
 {}+{}\e \int_{Q_t} (|\rho_{rr}|^2{}+{}m_{rr}|^2)\,drdt \notag \\
&&{}={}
 \int_{Q_t}\big(m_r\rho_{rr}{}+{}\frac{n-1}{r}m\rho_{rr}\big)\,drdt
 {}+{}\int_{Q_t}(\rho u^2{}+{}p_\delta)_rm_{rr}\,drdt \notag \\
 &&\quad {}+{}(n-1)\int_{Q_t}\big(\frac{\rho u^2}{r}m_{rr}
    - \frac{\e}{r}\rho_r\rho_{rr}\big)\,drdt
 {}-{}(n-1)\e\int_{Q_t}\big(\frac{m}{r}\big)_r m_{rr}\,drdt.\qquad\quad \label{est:higher-0}
 \end{eqnarray}

We now estimate the term $\int_{Q_T}(\rho u^2+p)_rm_{rr}\,drdt$ first. Consider
\begin{eqnarray}
&&\left| \int_{Q_t} p'_\delta (\rho)\rho_rm_{rr}\,drd\tau\right|\notag \\
&&\leq \Delta \int_{Q_t}|m_{rr}|^2\,drd\tau{}
+{}C_\Delta\int_{Q_t} \big(2\delta \rho+\kappa\gamma\rho^{\gamma-1}\big)^2 |\rho_r|^2\,drd\tau \notag \\
&&\leq \Delta \int_{Q_t}|m_{rr}|^2\,drd\tau{}
+{}C_\Delta\int_0^t\Big(\big(1+\|\rho(\tau,\cdot)\|_{L^\infty}^{2\gamma}\big)\int_a^b |\rho_r|^2\,dr\Big)d\tau,
\label{est:higher-0.1}
\end{eqnarray}
where $\Delta>0$ will be chosen later.

Consider $(\rho u^2)_rm_{rr}{}=u^2\rho_r m_r+ 2\rho u u_rm_{rr}$. We estimate
\begin{eqnarray*}
&&\int_{Q_t}|u^2\rho_r m_{rr}|\,drd\tau \notag\\
&&\leq \Delta\int_{Q_t} |m_{rr}|^2\,drd\tau
{}+{}
C_\Delta \int_0^t\Big( \|u(\tau,\cdot)\|^4_{L^\infty}\int_a^b |\rho_r(\tau,r)|^2\,dr\Big) d\tau \\
&&\leq \Delta\int_{Q_t} |m_{rr}|^2\,drd\tau +
C_\Delta \int_0^t\Big(\|u(\tau,\cdot)\|_{L^\infty}^4
\int_a^b h''_\delta(\rho)|\rho_r(\tau,r)|^2\,dr\Big) d\tau.
\end{eqnarray*}

Using the uniform estimates \eqref{est:u_max}, we obtain
\begin{equation}
\label{est:u-uniform-1}
\|u(\tau,\cdot)\|_{L^\infty(a,b)}^4{}\leq{} \|u\|_{L^\infty(Q_\tau)}^4
{}\leq{} C(\rhob,a,\|(\rho_0,u_0)\|_{L^\infty(a,b)})
\big(1{}+{}\| \rho\|_{L^\infty(Q_\tau)}^{2\max\{1, \gamma-1 \} }\big).
\end{equation}
Inserting this into the above inequality, we have
\begin{eqnarray*}
&&\int_{Q_t}|u^2\rho_rm_{rr}|\,drd\tau \notag\\
&&\leq{} \Delta\int_{Q_t} |m_{rr}|^2\,drd\tau
{}+{}
C_\Delta \int_0^t\Big((1+ \sup_{s\in[0,\tau]}\|\rho(s,\cdot)\|_{L^\infty}^{2\max\{1, \gamma-1\}})
\int_a^b h''_\delta(\rho)|\rho_r(\tau,r)|^2\,dr\Big) d\tau.
\end{eqnarray*}

On the other hand, using the estimate similar to \eqref{est:rho_bounded},  we can write
\begin{eqnarray}
\label{est:rho_uniform:1}
\|\rho(t,\cdot)\|_{L^\infty}^{\max\{4,\gamma+2\}}{}\leq C\Big(1+ \int_a^b |\rho_r(t,\cdot)|^2\,dr\Big)
\qquad\mbox{for $t\in[0,T]$}.
\end{eqnarray}

Using \eqref{est:rho_uniform:1} and $\gamma\in(1,3]$, we obtain
\begin{eqnarray}
&&\int_{Q_t}|u^2\rho_rm_{rr}|\,drd\tau \notag\\
&&\leq \Delta\int_{Q_t} |m_{rr}|^2\,drd\tau\notag \\
&&\quad +
C_\Delta \int_0^t\Big(\big(1+ \sup_{s\in[0,\tau]}\int_a^b |\rho_r(s,r)|^2 dr\big)
\int_a^b h''_\delta(\rho)|\rho_r(\tau,r)|^2\,dr\Big) d\tau.
\quad\label{est:high-2}
\end{eqnarray}

Furthermore, we have
\begin{eqnarray}
\int_{Q_t}|\rho uu_rm_{rr}|\,drd\tau &\leq& \Delta\int_{Q_t} |m_{rr}|^2\,drd\tau\notag \\
&&+ C_\Delta\int_0^t\Big(\|(\rho u^2)(\tau,\cdot)\|_{L^\infty}\int_a^b\rho(\tau,r)|u_r(\tau,r)|^2\,dr\Big) d\tau.
\label{2.14b}
\end{eqnarray}
Arguing as in  \eqref{est:u-uniform-1} and \eqref{est:rho_uniform:1}, we obtain
\begin{eqnarray}
\|(\rho u^2)(\tau,\cdot)\|_{L^\infty}
&\leq& C\Big(1+\sup_{s\in[0,\tau]}\|\rho(s,\cdot)\|_{L^\infty}^{\max\{2,\gamma\}}\Big) \notag \\
&\leq& C\Big(1+
\sup_{s\in[0,\tau]}\int_a^b |\rho_r(s,r)|^2\,dr\Big).
\label{est:rho_u^2}
\end{eqnarray}
Inserting this into \eqref{2.14b}, we obtain
\begin{eqnarray}
&&\int_{Q_t}|\rho uu_rm_{rr}|\,drd\tau \notag\\
&&\leq{} \Delta\int_{Q_t} |m_{rr}|^2\,drd\tau \notag\\
&&\quad
+C_\Delta \int_0^t\Big(\big(1+ \sup_{s\in[0,\tau]}\int_a^b|\rho_r(s,r)|^2 dr\big)
\int_a^b \rho(\tau,r)|u_r(\tau,r)|^2\,dr\Big) d\tau. \label{est:high-3}
\end{eqnarray}

Combining \eqref{est:higher-0.1}, \eqref{est:high-2}, and \eqref{est:high-3}, we obtain
\begin{eqnarray}
\left| \int_{Q_t} (\rho u^2 {}+{}p)_rm_{rr}\,drd\tau\right| &\leq&
\Delta\int_{Q_t} |m_{rr}|^2\,drd\tau \notag\\
&&+C_\Delta \int_0^t\Phi_1(\tau)\Big(1+ \sup_{s\in[0,\tau]}\int_a^b |\rho_r(s,r)|^2 dr\Big) d\tau,
\label{2.17} \notag
\end{eqnarray}
where
\[
\Phi_1(\tau){}={}\int_a^b
\Big(h''_\delta(\rho)|\rho_r(\tau,r)|^2{}+{}\rho(\tau,r)|u_r(\tau,r)|^2\Big)\,dr
\]
is an $L^1(0,T)$--function with the norm depending on $a,\e$, and $E_0$;
see \eqref{Energy_est} and \eqref{est:rho_int_2_gamma}.

Consider now
\begin{eqnarray*}
\left|\int_{Q_t}\frac{2\rho u^2}{r}m_{rr}drd\tau\right| &\leq& \Delta \int_{Q_t}|m_{rr}|^2\,drd\tau
{}+{}C_\Delta\int_0^t\Big(\|(\rho u^2)(\tau,\cdot)\|_{L^\infty}\int_a^b (\rho u^2)(\tau,r)\,dr\Big)d\tau\\
&\leq& \Delta \int_{Q_t}|m_{rr}|^2\,drd\tau
{}+{}C_\Delta\int_0^t \Big(1+ \sup_{s\in[0,\tau]}\int_a^b|\rho_r(s,r)|^2\,dr\Big)d\tau,
\end{eqnarray*}
where, in the last inequality, we have used \eqref{Energy_est} and \eqref{est:rho_u^2}.
All the other terms in \eqref{est:higher-0} can be estimated by similar arguments.
Thus, we obtain
\begin{eqnarray*}
&&\sup_{\tau\in[0,t]}\int_a^b \big(|\rho_r(\tau,s)|^2{}+{}|m_r(\tau,s)|^2\big)\,dr{}
+{}\e\int_{Q_t}\big(|\rho_{rr}|^2{}+{}|m_{rr}|^2\big)\,drd\tau \\
&&\leq \Delta\int_{Q_t}\big(|\rho_{rr}|^2{}+{}|m_{rr}|^2\big)\,drd\tau \\
&&\quad {}+{}
C_\Delta\int_0^t\big(1+\Phi(\tau)\big)
\Big(1+\sup_{s\in[0,\tau]}\int_a^b\big(|\rho_r(s,r)|^2{}+{}|m_r(s,r)|^2\big)\,dr\Big)\,d\tau,
\end{eqnarray*}
where $\Phi(\tau)=\Phi_1(\tau)+ \|\rho(\tau,\cdot)\|_{L^\infty}^{2\max\{2, \gamma\}}$.

Choosing $\Delta$ small enough and using the Gronwall-type argument and Lemma \ref{2.1-b},
we complete the proof.
\end{proof}

As a corollary, we can first bound $\|\rho\|_{L^\infty(Q_T)}$,
which follows directly from \eqref{2.10a} and \eqref{est:rho_uniform:1}, and
then bound $\|u\|_{L^\infty(Q_T)}$ from Lemma \ref{lemma:uniform}.

\begin{lemma}
\label{lemma:2.4}
There exists an a priori bound for $\|(\rho,u)\|_{L^\infty(Q_T)}$
in terms of the parameters $T,E_0, \|(\rho_0,u_0)\|_{L^\infty(a,b)}$, and
$\|(\rho_0,u_0)\|_{H^1(a,b)}$.
\end{lemma}

Define
\[
 \phi(\rho){}={}\left\{
\begin{array}{ll}
\frac{1}{\rho}-\frac{1}{\rhot} {}+{}\frac{\rho-\rhot}{\rhot^2},& \rho<\rhot,\\[2mm]
0, & \rho>\rhot.
\end{array}
 \right.
\]

\begin{lemma}
There exists $C>0$ depending on $\|\phi(\rho_0)\|_{L^1(a,b)}$
and the other parameters of the problem
such that
\begin{eqnarray}
\label{est:rho_r_integral}
\sup_{t\in[0,T]}\int_a^b\phi(\rho(t,\cdot))\,dr
{}+{}\int_{Q_T}\frac{|\rho_r|^2}{\rho^3}\,drdt{}\leq{} C.
\end{eqnarray}
\end{lemma}

\begin{proof}
Indeed, multiplying the first equation in \eqref{eq:NS} by $\phi'(\rho)$, we have
\begin{eqnarray*}
&&\phi_t{}+{}(u\phi)_r{}-{}\e \, \phi_{rr}{}+{}(n-1)\e \frac{|\rho_r|^2}{\rho^3}\chi_{\{\rho<\rhot\}}\\
&&={}2\big( \frac{1}{\rho}-\frac{1}{\rhot}\big)u_r\chi_{\{\rho<\rhot\}}{}+{}
\frac{n-1}{r}\rho u\big(\frac{1}{\rho^2}-\frac{1}{\rhot^2}\big)\chi_{\{\rho<\rhot\}}
+\frac{(n-1)\e}{r}\big(\frac{1}{\rho^2}-\frac{1}{\rhot^2}\big)\rho_r\chi_{\{\rho<\rhot\}}.
\end{eqnarray*}
Integrating the above equation in $(t,r)$ and using the boundary conditions \eqref{eq:bc},
we have
\begin{eqnarray}
&&\sup_{t\in[0,T]}\int_a^b\phi(\rho)\,dr
{}+{}\e (n-1)\int_{Q_T\cap\{\rho<\rhot\}}\frac{|\rho_r|^2}{\rho^3}\,drdt \notag \\
{}&&\leq{} \left|\int_{Q_T\cap\{\rho<\rhot\}}2\big(\frac{1}{\rho}-\frac{1}{\rhot}\big)u_r\,drdt \right|
{}+{}\left|\int_{Q_T\cap \{\rho<\rhot\}}\frac{n-1}{r}\rho u\left(\frac{1}{\rho^2}-\frac{1}{\rhot^2}\right)\,drdt \right| \notag \\[1mm]
{}&&\quad +{}\left|\int_{Q_T\cap\{\rho<\rhot\}} \frac{(n-1)\e}{r}\rho_r\left(\frac{1}{\rho^2}-\frac{1}{\rhot^2}\right)\,drdt \right| \notag \\
\label{est_00}
{}&&={}I_1{}+{}I_2{}+{}I_3.
\end{eqnarray}
Integrating by parts, we have
\begin{eqnarray}
 I_1{}\leq{}2\int_{Q_T\cap\{\rho<\rhot\}} \left|\frac{\rho_r u}{\rho^2}\right|
 {}\leq{}\frac{\e}{8}\int_{Q_T\cap \{\rho<\rhot\}}\frac{|\rho_r|^2}{\rho^3}\,drdt
 {}+{}C_\e\int_{Q_T\cap \{\rho<\rhot\}}\frac{|u|^2}{\rho}\,drdt. \notag
 \end{eqnarray}
Since $\rho^{-1}\leq\phi(\rho)$ for small $\rho$, $u$ is bounded in $L^\infty$,
and $|\{\rho(t,\cdot)\leq\rhot\}|$ is bounded independently of $T$,
then the last term in the above inequality is bounded by
\[
C\Big(1{}+{}\int_{Q_T}\phi(\rho)\,drdt\Big).
\]
Thus, we have
\begin{eqnarray}
I_1&\leq&2\int_{Q_T\cap\{\rho<\rhot\}} \Big|\frac{\rho_r u}{\rho^2}\Big|\,drdt\notag\\
&\leq& \Delta\int_{Q_T\cap \{\rho<\rhot\}}\frac{|\rho_r|^2}{\rho^3}\,drdt
{}+{} C_\Delta\left(1+\int_{Q_T}\phi(\rho)\,drdt\right).
\end{eqnarray}
Also, by the similar arguments,
\begin{eqnarray}
 I_2{}={}\left|\int_{Q_T\cap\{\rho<\rhot\}}\frac{n-1}{r}\left(\frac{\rho u}{\rhot^2}
       {}-{}\frac{u}{\rho}\right)\,drdt{}\right|
       \leq{} C\left(1{}+{}\int_{Q_T}\phi(\rho)\,drdt\right),
\end{eqnarray}
and
\begin{eqnarray}
 I_3&\leq & C\int_{Q_T\cap \{\rho<\rhot\}}\left| \frac{\e\rho_r}{\rho^2} \right|\,drdt\notag\\
 &\leq&\Delta\int_{Q_T\cap \{\rho<\rhot\}}\frac{|\rho_r|^2}{\rho^3}\,drdt
 + C_\Delta\left(1+\int_{Q_T}\phi(\rho)\,drdt\right).
\end{eqnarray}
Combining the last three estimates in \eqref{est_00}, choosing $\Delta>0$ sufficiently small,
and using the Gronwall-type inequality, we obtain
the {\it a priori} estimate we need.
\end{proof}

Then we have the following estimate:
\begin{eqnarray}
 \label{est:v_int}
\int_0^T\Big\|\frac{1}{\rho(t,\cdot)}\Big\|_{L^\infty(a,b)}\,dt
{}&\leq&{}C\left(1{}+{}\big(\int_{Q_T}\frac{|\rho_r|^2}{\rho^3}\,drdt\big)^{\frac{1}{2}}
\big(\int_{Q_T}\phi(\rho)\,drdt\big)^{1/2}\right) \notag \\
{}&\leq&{}C \left(1{}+{}\big(\int_{Q_T}\frac{|\rho_r|^2}{\rho^3}\,drdt\big)^{\frac{1}{2}}\right).
\end{eqnarray}

\begin{lemma}
There exists $C$ depending on $\|\phi(\rho_0)\|_{L^1(a,b)}$ and the other parameters as in Lemma {\rm 2.4}
such that
\begin{equation}
\label{est:u_r_max}
 \int_0^T\Big\|(\frac{m_r}{\rho}, \frac{\rho_r}{\rho}, u_r)(t,\cdot)\Big\|_{L^\infty(a,b)}\,dt{}\leq{}C,
\end{equation}
and
\begin{equation}\label{rho-lower-bound}
C^{-1}\le \rho(t,r)\le C.
\end{equation}
\end{lemma}

\begin{proof}
Indeed, by the Sobolev embedding and \eqref{est:v_int}, we have
\begin{eqnarray*}
\int_0^T\Big\|\frac{m_r(t,\cdot)}{\rho(t,\cdot)}\Big\|_{L^\infty(a,b)}\,dt
{}&\leq&{}\int_0^T\|m_r(t,\cdot)\|_{L^\infty(a,b)}\|\rho^{-1}(t,\cdot)\|_{L^\infty(a,b)}\,dt \notag \\
& \leq &  C  \int_0^T  \Big(\int_a^b|m_{rr}|^2\,dr\Big)^{\frac{1}{2}}
\Big(1{}+{}\big(\int_a^b\frac{|\rho_r|^2}{\rho^3} \,dr\big)^{\frac{1}{2}}\Big)\,dt,
\end{eqnarray*}
which is bounded by \eqref{2.10a} and \eqref{est:rho_r_integral}.
The estimate for $\frac{\rho_r}{\rho}$ is the same. The estimate for $u_r$ follows
from $u_r{}={}\frac{m_r}{\rho}{}-{}\frac{u\rho_r}{\rho}$, the estimates above, and Lemma \ref{lemma:2.4}.

Now we can obtain a uniform estimate for $v{}={}\frac{1}{\rho}$. Notice that $v$ verifies the inequality:
\begin{equation*}
 v_t{}+{}\big(u-\frac{\e(n-1)}{r}\big)v_r{}-{}\e v_{rr}{}\leq{} \big(u_r{}+{}\frac{(n-1)u}{r}\big)v.
\end{equation*}
By the maximum principle, we have
\begin{equation}
 \label{est:v_max}
\max_{Q_T}v
{}\leq{}C\max\{\|v_0\|_{L^\infty(a,b)},\bar{v}\} e^{C\int_0^T\|(u_r, u)(\tau,\cdot)\|_{L^\infty(a,b)}\,d\tau}
{}\leq{}C\max\{\|v_0\|_{L^\infty(a,b)},\bar{v}\},
\end{equation}
by Lemma \ref{lemma:2.4} and \eqref{est:u_r_max}.
\end{proof}

\smallskip
The estimates in Lemma \ref{lemma:2.4} and \eqref{est:v_max}
are the required {\it a priori} estimates.
The proof of Theorem \ref{theorem:exist} is completed.
\end{subsection}
\end{section}

\medskip
\begin{section}{Proof of Theorem \ref{main}}

In this section, we provide a complete proof of Theorem \ref{main}.
As indicated earlier,
the constant $M$ is a universal constant, {\it independent of $\e>0$},
from now on.

\begin{subsection}{A Priori Estimates Independent of $\e$}

We will need the following estimate.
\begin{lemma}
\label{claim:local_integrability}
Let  $l=0,\cdots, n-1$, and $a_1\in(a,1]$.  There exists $M{}={}M(\gamma,a_1,E_0)$
such that,
for any $T>0$,
\begin{equation}
\sup_{t\in[0,T]}\int_{a_1}^b \rho(t,\cd)^\gamma\,r^ldr
{}\leq M\big(1{}+{}\bar{\rho}^\gamma b^n\big).
\end{equation}
\end{lemma}

\begin{proof} The proof is based on the energy estimate \eqref{Energy_est}. Let
$$
\hat{e}(\rho){}={}\rho^\gamma -\rhob^\gamma-\gamma\rhob^{\gamma-1}(\rho-\rhob).
$$
Using the Young inequality, we find that there exists $M(\gamma)>0$ such that
 $$
 \rho^\gamma\leq M(\gamma)\big(\hat{e}(\rho){}+{}\rhob^\gamma\big).
 $$
Then we have
\begin{equation*}
\int_{a_1}^b\rho^\gamma\,r^ldr{}
\leq{} M\Big(\int_{a_1}^b \hat{e}(\rho)\,r^ldr{}+{}\bar{\rho}^\gamma b^{l+1}\Big).
\end{equation*}
Since $0<a(\e)<1<b(\e)<\infty$, we have
\[
\int_{a_1}^b \hat{e}(\rho(t,r))\,r^ldr{}
\leq{}{a_1}^{l+1-n}\sup_{\tau\in[0,t]}\int_a^b \bar{\eta}^*(\rho(\tau,r),m(\tau,r))\,r^{n-1}dr{}\leq{}{a_1}^{1-n}E_0,
\]
by Proposition \ref{2.1} for $E_0$, independent of $\e$,
which implies that, for all $l=0, \cdots, n-1$,
\begin{equation*}
\int_{a_1}^b\rho^\gamma\,r^ldr{}
\leq{} M\big(a_1^{1-n} E_0 +{}\bar{\rho}^\gamma b^{n}\big)
\le M\big( 1+{}\bar{\rho}^\gamma b^{n}\big).
\end{equation*}
\end{proof}

\begin{lemma}
\label{lemma 3.2}
There exists $M=M(T)$, independent of $\e$, such that
\begin{equation}
\label{est:delta}
\int_0^T\int_r^b\rho^3\,y^{n-1}dydt{}\leq{}M\Big(1+ \frac{ b^n}{\e}\Big)\qquad \mbox{for any $r\in(a,b)$}.
\end{equation}
\end{lemma}

\begin{proof}
Consider first the case $\gamma\in(1,2)$. We estimate
\begin{eqnarray*}
\e\int_0^T\int_r^b\rho^3{}\, y^{n-1}dydt{}
&\leq&{}M\e\int_0^T\sup_{(r,b)}\rho^{3-\gamma}(t,\cdot)\,dt \notag \\
&\leq& M +M\e\int_0^T\int_r^b\rho^{3-\frac{3\gamma}{2}}|(\rho^{\frac{\gamma}{2}})_y|\,dydt   \notag \\
&\leq& M+M\e\int_0^T\int_r^b\rho^{6-3\gamma}(y^{n-1})^{-1}\,dydt\\
{}&=&{}
M+M\e\int_0^T\int_r^b \rho^{6-3\gamma} (y^{n-1})^{2-\gamma}(y^{n-1})^{\gamma-3}\,dydt\notag \\
&\leq& M{}+{}\frac{\e}{2}\int_0^T\int_r^b\rho^3\,y^{n-1}dydt,\notag
\end{eqnarray*}
where, in the last inequality, we have used the Jensen inequality. It follows from the above computation that
\[
\e\int_0^T\int_r^b \rho^3\,y^{n-1}dydt{}\leq{}M(T)\qquad \mbox{for all $r\in(a,b)$},
\]
which arrives at \eqref{lemma 3.2}.

Let now $\gamma\in[2,3]$. First, we notice that
\begin{eqnarray*}
\sup_{t\in[0,T]} \int_r^b \rho\,y^{n-1}dy
&\leq& \sup_{t\in[0,T]} \left( \int_a^b \rho^\gamma\,r^{n-1}dr \right)^{\frac{1}{\gamma}}
 \left(\int_r^b y^{n-1}dy\right)^{\frac{\gamma-1}{\gamma}}\\
&\leq & Mb^{\frac{n(\gamma-1)}{\gamma}}\le M b^n
\end{eqnarray*}
since $b>1$.

Then we argue as above:
\begin{eqnarray*}
\int_0^T\int_r^b \rho^3\,y^{n-1}dydt
&\leq& \int_0^T\Big(\sup_{(r,b)}\rho^2(t,\cdot)\int_r^b\rho\,y^{n-1}dy\Big) dt \\
&\leq& M b^{n}\left(1{}+{}\int_0^T\int_r^b \rho|\rho_r|\,dydt\right)\\
&=& M b^{n}\left(1{}
  +{}\int_0^T\int_r^b  \rho^{2-\frac{\gamma}{2}}\rho^{\frac{\gamma-2}{2}}|\rho_y|\,dydt\right)\\
&\leq&  M b^n\left(1{}+{}\frac{1}{\e}{}+{}\int_0^T\int_r^b\rho^{4-\gamma}(y^{n-1})^{-1}\,dydt\right)\\
&\leq&  M b^n\Big(1+\frac{1}{\e}\Big)\\
&\leq&  M \frac{b^n}{\e},
\end{eqnarray*}
where, in the last inequality, we have used the Jensen inequality
with powers $\frac{\gamma}{4-\gamma}$ and $\frac{\gamma}{2\gamma-4}$
and the energy estimate \eqref{Energy_est}.
\end{proof}

\begin{lemma}\label{lemma:3.2} Let $K$ be a compact subset of $(a,b)$.
Then, for $T>0$, there exists $M{}={}M(K,T)$ independent of $\e$ such that
\begin{equation}
\label{HI_est_1}
\int_0^T\int_K(\rho^{\gamma+1}+\delta \rho^3)\,drdt{}\leq{}M.
\end{equation}
\end{lemma}

\begin{proof} We divide the proof into five steps.

\smallskip
1. Let $\omega(r)$ be a smooth positive, compactly supported function on $(a,b)$.
We multiply the momentum equation in \eqref{eq:NS} by $\omega$ to obtain
\begin{eqnarray}\label{3.3}
&&(\rho u \omega)_t{}+{}\big((\rho u^2+p_\delta)\omega\big)_r
{}+{}\frac{n-1}{r}\rho u^2\omega{}-{}\e\big(\omega(m_r{}+{}\frac{n-1}{r}m)\big)_r \notag \\
&&{}={}\big(\rho u^2+p_\delta{}-{}\e(m_r{}+{}\frac{n-1}{r}m)\big)\omega_r.
\end{eqnarray}
Integrating \eqref{3.3} in $r$ over $(r,b)$ yields
\begin{eqnarray}\label{3.4}
\Big(\int_r^b\rho u\omega\,dy\Big)_t
{}+{}\int_r^b\frac{n-1}{y}\rho u^2\omega\,dy
{}+{}\e\omega\big(m_r{}+{}\frac{n-1}{r}m\big)
{}={}\omega(\rho u^2{}+{}p_\delta){}+{}f_1,
\end{eqnarray}
where
\[
f_1{}={}\int_r^b\big(\rho u^2 +p_\delta -\e(m_y{}+{}\frac{n-1}{y}m)\big)\omega_y\,dy.
\]

\smallskip
2. Multiplying \eqref{3.4} by $\rho$ and using the continuity equation \eqref{eq:NS}, we have
\begin{eqnarray*}
&&\Big(\rho\int_r^b\rho u\omega\,dy\Big)_t {}+{}\Big((\rho u)_r{}+{}\frac{n-1}{r}\rho m
{}-{}\e(\rho_{rr}{}+{}\frac{n-1}{r}\rho_r)\Big)\int_r^b\rho u\omega\,dy \notag \\
&&\quad +\rho\int_r^b\frac{n-1}{y}\rho u^2\omega\,dy{}+{}\e\rho\omega\big(m_r{}+{}\frac{n-1}{r}m\big)\\
&&{}={}(\rho^2 u^2 + \rho p_\delta)\omega{}+{}\rho f_1,
\end{eqnarray*}
and
\begin{eqnarray}
&&\Big(\rho\int_r^b\rho u\omega\,dy\Big)_t{}+{}\Big(\rho u\int_r^b\rho u\omega\,dy\Big)_r\notag \\
&&{}+{}\e\Big( -\big(\rho_{rr}{}+{}\frac{n-1}{r}\rho_r\big)\int_r^b\rho u\omega\,dy
{}+{}\rho\omega\big(m_r{}+{}\frac{n-1}{r}m\big)\Big)\notag\\
&&{}={}\rho p_\delta \omega {}+{}f_2,
\end{eqnarray}
where
\[
f_2{}={}\rho f_1{}-{}\frac{n-1}{r}\rho m\int_r^b\rho u\omega\,dy-\rho\int_r^b\frac{n-1}{y}\rho u^2\omega dy.
\]
Notice that
\begin{eqnarray*}
&&-\big(\rho_{rr}{}+{}\frac{n-1}{r}\rho_r\big)\int_r^b\rho u\omega\,dy
{}+{}\rho\omega\big(m_r{}+{}\frac{n-1}{r}m\big)\notag\\
&&={} -\Big(\rho_r\int_r^b\rho u \omega\,dy\Big)_r{}-{}\rho u\rho_r\omega
{}-{}\left(\frac{n-1}{r}\rho\int_r^b\rho u \omega\,dy\right)_r{}-{}\frac{n-1}{r}\rho^2 u\omega \notag \\
{}&&\quad +{}\frac{n-1}{r^2}\rho\int_r^b\rho u \omega\,dy{}
  +{}\rho^2 u_r\omega
{}+{}\rho u\rho_r\omega {}+{}\frac{n-1}{r}\rho^2 u\omega \notag \\
{}&&={}
-\big(\rho\int_r^b\rho u\omega\,dy\big)_{rr}{}-{}(\rho^2 u\omega)_r{}
-{}\left(\frac{n-1}{r}\rho\int_r^b\rho u\omega\,dy\right)_r \notag \\
{}&&\quad +{}\rho^2 u_r\omega{}+{}\frac{n-1}{r^2}\rho\int_r^b\rho u\omega\,dy.
\end{eqnarray*}
It then follows that
\begin{eqnarray}\label{3.6a}
&&\Big(\rho\int_r^b\rho u\omega\,dy\Big)_t{}+{}\Big(\rho u\int_r^b\rho u\omega\,dy\Big)_r{}
-{}\e\Big(\rho\int_r^b\rho u\omega\,dy\Big)_{rr}{}-{}\e(\rho^2 u\omega)_r \notag \\
&&\quad{}-{}\e\left(\frac{n-1}{r}\rho\int_r^b\rho u\omega\,dy\right)_r
{}+{}\e\rho^2 u_r\omega\notag\\
&&{}={}p_\delta\rho\omega{}+{}f_3,
\end{eqnarray}
where
$f_3{}={}f_2{}-{}\e\frac{n-1}{r^2}\rho\int_r^b\rho u\omega\,dy$.

\smallskip
3. We multiply \eqref{3.6a} by $\omega$ to obtain
\begin{eqnarray}
&&\Big(\rho\omega\int_r^b\rho u\omega\,dy\Big)_t{}
+{}\Big(\rho u\omega\int_r^b\rho u\omega\,dy\Big)_r{}
-{}\e\Big(\omega\big(\rho\int_r^b\rho u\omega\,dy\big)_{r}\Big)_r {}\notag \\
&&\quad  +{}\e\Big(\rho\omega_r\int_r^b\rho u\omega\,dy\Big)_r
 {}-{}\e(\rho^2 u\omega^2)_r{}-{}\e\left(\frac{n-1}{r}\rho\omega\int_r^b\rho u\omega\,dy\right)_r
 \notag\\
&&\quad +{}\e\rho^2u_r\omega^2{}+{}\e\rho^2 u\omega \omega_r\notag\\[2mm]
&&={}p_\delta\rho\omega^2{}+{}f_4, \label{3.7a}
\end{eqnarray}
where
$f_4{}={} \omega f_3{}+{}\rho u\omega_r\int_r^b\rho u\omega\,dy{}
{}-{}\frac{n-1}{r}\rho\omega_r\int_r^b\rho u \omega\,dy$.

\smallskip
We integrate \eqref{3.7a} over $[0,T]\times[a,b]$ to obtain
\begin{eqnarray}
&&\int_{Q_T} \big(\delta\rho^3+\kappa\rho^{\gamma+1}\big)\omega^2\,drdt{}\notag\\
&&={} \int_{Q_T}\big(\e\rho^2u_r\omega^2{}+{}\e\rho^2 u\omega\omega_r\big)\,drdt \notag \\
&&\quad +{}\int_a^b\Big(\rho\omega\int_r^b\rho u\omega\,dy\Big)\Big|^T_0\,dr
{}-{}\int_{Q_T} f_4\,drdt \notag \\
&&\leq{} \e\int_{Q_T}\rho^3\omega^2\,drdt{}
+{}\e M\int_{Q_T}\big(\rho|u_r|^2\omega^2{}+{}\rho|u|^2|\omega_r|^2\big)\,drdt \notag \\
&&\quad +{}\int_a^b\Big(\rho\omega\int_r^b\rho u\omega\,dy\Big)\big|^T_0\,dr {}
-{}\int_{Q_T} f_4\,drdt \notag \\
\label{eq:current-1}
&&\leq \e\int_{Q_T}\rho^3\omega^2\,drdt{}+{}M(\supp\,\omega,T,E_0).
\end{eqnarray}
The last inequality follows easily from \eqref{Energy_est}--\eqref{est:energy_2}
and the formula for $f_4$.

\medskip
4. Claim:
There exists $M=M(\mbox{supp}\,\omega,T,E_0)$ such that
\begin{equation}\label{3.8a}
\e\int_{Q_t} \rho^3\omega^2\, drdt{}\leq{} M{}+{}M\e\int_{Q_t}\rho^{\gamma+1}\omega^2\,drdt.
\end{equation}

\medskip
If $\gamma\geq 2$, the claim is trivial.
Let $\gamma{}<{}\beta{}\leq{} 3$. We estimate
\begin{eqnarray}\label{3.11}
&&\e\int_{Q_T}\rho^\beta\omega^2 dx dt\\
&&\leq {}\e\sup_{\supp\,\omega}\big(\rho^{\beta-\gamma}\omega^2\big)
\int_{Q_T\cap\, {\supp\,\omega}}\rho^\gamma{}drdt \notag\\
&&\leq{} \e M\sup_{\supp\,\omega}\big(\rho^{\beta-\gamma}\omega^2\big)\notag \\
&&\leq{}\e M\int_{Q_T}\rho^{\beta-\gamma-\frac{\gamma}{2}}|(\rho^{\frac{\gamma}{2}})_r|\omega^2{}drdt
+{}\e M\int_{Q_T} \rho^{\beta-\gamma}\omega|\omega_r| drdt\notag \\
&&\leq \e M\Big(\int_{Q_T\cap\, {\supp\,\omega}}\rho^\gamma drdt {}
+{} \int_{Q_T}|(\rho^{\frac{\gamma}{2}})_r|^2\omega^2drdt{}
+{}\int_{Q_T} \rho^{2\beta-3\gamma}\omega^2drdt\Big)\notag \\
&&\leq M\Big( 1{}+{}\e \int_{Q_T}\rho^{2\beta-3\gamma}\omega^2 drdt\Big).
\end{eqnarray}

If $2\beta-3\gamma\leq\gamma+1$, the estimate of the claim follows.
Otherwise, since $2\beta-3\gamma<\beta$ (note that $\beta\leq 3$),
we can iterate \eqref{3.11} with $\beta$ replaced by $2\beta-3\gamma$
and improve \eqref{3.11}:
\begin{eqnarray}
\e\int_{Q_T}\rho^\beta\omega^2 drdt
{}\leq{}  M\Big(1 {}+{}\e \int_{Q_T}\rho^{4\beta-9\gamma}\omega^2 drdt\Big).
\end{eqnarray}
If $4\beta-9\gamma$ is still larger than $\gamma+1$,
we iterate the estimate again.
In this way, we obtain
a recurrence relation $\beta_n{}={}2\beta_{n-1}-3\gamma$, $\beta_0=\beta\leq3$, and the estimate
\[
\e\int_{Q_T}\rho^\beta\omega^2 drdt
{}\leq{}  M(n)\Big(1 {}+{}\e \int_{Q_T} \rho^{\beta_n\gamma}\omega^2drdt \Big).
\]
Solving the recurrence relation, we obtain
$$
\beta_n{}={}2^n\beta{}-{}3\gamma(2^{n-1}-1).
$$
For some $n$, the expression is less than $\gamma+1$
(note that $\beta\leq 3$).
Then the expected estimate is obtained.

\medskip
5. Now returning to \eqref{eq:current-1}, we have
\[
\int_{Q_T}\big(\rho^{\gamma+1}+\delta\rho^2\big)\omega^2 drdt {}\leq{}M(\supp\,\omega, T, E_0)
\]
for all small  $\e>0$.
\end{proof}

The following lemma holds for weak entropies $\eta$ (also {\it cf.} \cite{DiPerna}).
\begin{lemma}
\label{Energy_control_lemma}
 Let $\eta^*(\rho,m)$ be the mechanical energy of system \eqref{eq:Eu},
 and let $(\eta_\psi,q_\psi)$ be an entropy pair
 \eqref{eta}--\eqref{q} with the generating function $\psi(s)$ satisfying
\[
\sup_s|\psi''(s)|{}<{}\infty.
\]
Then,
for any $(\rho,m){}\in{}\mathbb{R}^2_+$ and any vector $\bar{a}{}={}(a_1,a_2)$,
\begin{equation}
|\bar{a}\grad^2\eta\bar{a}^\top|{}\leq{} M_\psi\,  \bar{a}\grad^2\eta^*\bar{a}^\top
\qquad \mbox{for some $M_\psi>0$}.
\end{equation}
\end{lemma}

\begin{lemma}\label{HI:2}
Let $K\subset (a,b)$ be compact.
There exists $M{}={}M(K,T)$ independent of $\e$
such that, for any $\e>0$,
\begin{align*}
\int_0^T\int_K\big(\rho|u|^3{}+{}\rho^{\gamma+\theta}\big)\,drdt
{}\leq{}M\big(1+\rhob^\gamma b^n
{}+{}\frac{\delta}{\e}b^n\big).
\end{align*}
\end{lemma}

\begin{proof} We divide the proof into five steps.

\smallskip
1. Let $(\etach,\qch)$ be an entropy pair corresponding to $\psi(s){}={}\frac{1}{2}s|s|$.
Define
\begin{eqnarray*}
&&\etat(\rho,m){}={}\etach(\rho,m){}-{}\grad_{(\rho,m)}\etach(\rhob,0)\cdot(\rho-\rb,m)\ge 0,\\
&&\qt(\rho,m){}={}\qch(\rho,m){}-{}\grad_{(\rho,m)}\etach(\rhob,0)\cdot(m,\frac{m^2}{\rho}{}+{}p).
\end{eqnarray*}
Note that the entropy pair $(\etach, \qch)$ is defined for system \eqref{eq:Eu} with
pressure $p{}={}\kappa\rho^\gamma$, rather than $p_\delta$.
Then $(\etat,\qt)$ is still an entropy pair of \eqref{eq:Eu}.

We multiply the continuity equation in \eqref{eq:NS} by $\etat_\rho r^{n-1}$,
the momentum equation \eqref{eq:NS} by $\etat_m r^{n-1}$, and then add them to obtain
\begin{eqnarray}
\label{eq:int_ii_1}
&&(\etat r^{n-1})_t{}+{}(\qt r^{n-1})_r{}
+{}(n-1)r^{n-2}\big(-\qch{}+{}m\etach_\rho{}+{}\frac{m^2}{\rho}\etach_m{}+{}\etach_m(\rhob,0) p(\rho)\big) \notag \\
&&{}={} \e r^{n-1}\Big((\rho_{rr}{}+{}\frac{n-1}{r}\rho_r)\etat_\rho{}+{}(m_r{}+{}\frac{n-1}{r}m)_r\etat_m\Big){}-{}(\delta\rho^2)_r\etat_mr^{n-1}.
\end{eqnarray}

\smallskip
2. It can be checked directly  that, for some constant $M=M(\gamma)>0$,
\begin{eqnarray}
\label{id:0.1}
&&\qt(\rho,m){}\geq{} \frac{1}{M}(\rho|u|^3{}+{}\rho^{\gamma+\theta})
{}-{}M(\rho{}+{}\rho|u|^2{}+{}\rho^\gamma),\\[2mm]
\label{id:0.2}&&-\qch{}+{}m(\etach_\rho{}+{}u\etach_m){}\leq{}0,\\[2mm]
&&|\etach_m|{}\leq{}M\big(|u|{}+{}\rho^\theta\big), \quad |\etach_\rho|{}\leq{}M \big(|u|^2{}+{}\rho^{2\theta}\big),\\[2mm]
&&|\etat|{}\leq{}M\big(\rho+\rho|u|^2+\rho^\gamma\big),
\quad  \rho|\etat_\rho+u\etat_m|{}\leq{}M\big(\rho+\rho|u|^2+\rho^\gamma\big),
\end{eqnarray}
and, for $\etach_\rho{}+{}u\etach_m$  considered as a function of $(\rho,u)$,
\begin{eqnarray}
|\big(\etach_\rho{}+{}u\etach_m\big)_\rho|{}\leq{}M\big(\rho^{\theta-1}|u|{}+{}\rho^{2\theta-1}\big),
\qquad |\big(\etach_\rho{}+{}u\etach_m\big)_u|{}\leq{} M\big(|u|{}+{}\rho^\theta\big).
\end{eqnarray}
Also see \cite{CP} for these inequalities.

\smallskip
Moreover, note that, at $r=b$,
\begin{equation}
\label{id:0.4}
\qt(\rhob,0){}={}\qch(\rhob,0){}={}c_0(\gamma)\rhob^{\gamma+\theta},
\quad \;|\etach_m(\rhob,0)|{}={}c_1(\gamma)\rhob^\theta,\;\quad
\etach_\rho(\rhob,0){}={}0,
\end{equation}
for some positive $c_i(\gamma), i=0,1$, depending only on $\gamma$.

\medskip
3. We integrate equation \eqref{eq:int_ii_1} over $(0,T)\times(r,b)$
to find
\begin{eqnarray}
\int_0^T\qt(\tau,r)r^{n-1}\,d\tau{}
&=&{}c(\theta)\rhob^{\gamma+\theta}b^{n-1}T{}+{}\int_r^b\left(\etat(T,y){}-{}\etat(0,y)\right)\,y^{n-1}dy \notag \\
{}&&+{}(n-1)\int_0^T\int_r^b\left( -\qch{}+{}m\etach_\rho{}+{}\frac{m^2}{\rho}\etach_m\right)\,y^{n-2}dyd\tau \notag \\
&&{}+{}(n-1)\int_0^T\int_r^by^{n-2} \etach_m(\rhob,0)\big(p(\rho)-p(\rhob)\big)\,dyd\tau  \notag\\
{}&&+{}\int_0^T\int_r^b\e y^{n-1}\Big( (\rho_{yy}+\frac{n-1}{y}\rho_y)\etat_\rho{}+{}(m_y{}+{}\frac{n-1}{y}m)_y\etat_m\Big)\,dyd\tau \notag \\
&&+ \int_0^T\int_r^b \delta\rho^2\big( (\etat_m)_\rho\rho_y{}+{}(\etat_m)_u u_y\big)y^{n-1}\,dyd\tau\notag \\
{}&&+{}(n-1)\int_0^T\int_r^b \delta \rho^2\etat_m y^{n-2}\,dyd\tau\notag \\
\label{eq:nash-1}
{}&=&{}I_1{}+{}\cdots{}+{}I_7.
\end{eqnarray}

4. Now we estimate the terms in \eqref{eq:nash-1}.
Clearly,
$$
|I_1|{}\leq{}M \rhob^{\gamma+\theta}b^{n-1}\leq M \rhob^{\gamma}b^n,
$$
since $\rhob<1$ and $b>1$ for small $\e>0$.

\medskip
Notice that $|\etat(\rho,m)|{}\leq{} \eta^*(\rho,m)$.
It then follows that
\begin{eqnarray*}
|I_2|& \leq &
\int_r^b|\etat(\rho(T,r),m(T,r))|\,r^{n-1}dr\\
&\leq &
\int_a^b \eta^*(\rho(T,r),m(T,r))\,r^{n-1}dr.
\end{eqnarray*}
By the energy estimate \eqref{est:energy_2}, $|I_2(t,r)|{}\leq{}E_0$.

\medskip
The term $I_3$ is nonpositive by \eqref{id:0.2} and can be dropped.

\medskip
Using Step 2, we have
\begin{equation}
|I_4(t,r)|{}\leq{} M(a_1,T)\big(1{}+{}\rhob^\gamma b^n\big) \qquad \mbox{for any}\,\, (t,r)\in[0,T]\times [a_1,b].
\end{equation}

\medskip
5. Consider $I_5$. We write
\begin{eqnarray*}
&&r^{n-1}(\rho_{rr}{}+{}\frac{n-1}{r}\rho_r)\etat_\rho{}={}(r^{n-1}\rho_r)_r\etat_\rho,\\
&&r^{n-1}(m_r{}+{}\frac{n-1}{r}m)_r\etat_m{}={}(r^{n-1}m_r)_r\etat_m{}-{}(n-1)r^{n-3}m\etat_m,
\end{eqnarray*}
and employ integration by parts (note that $\etat_\rho(\rhob,0){}={}\etat_m(\rhob,0){}={}0$) to obtain
\begin{eqnarray}
I_5{}&=&{}-\e\int_0^t\int_r^b\big(\rho_y (\etat_\rho)_y{}
+{}m_y(\etat_m)_y\big)\,y^{n-1}dyd\tau{}
-{}(n-1)\e\int_0^t\int_r^b m\etat_m\,y^{n-3}dyd\tau \notag \\
& &+ \e \int_0^t\etat_r(\tau,r)\,r^{n-1}d\tau\notag\\
&=& J_1{}+{}J_2{}+{}J_3.
\end{eqnarray}

Using the energy estimate \eqref{Energy_est}
and  Lemma \ref{Energy_control_lemma}, we have
$$
|J_1(t,r)|{}\leq{}M E_0.
$$

Also, using Step 2 and \eqref{id:0.4}, we have
\[
|m\etat_m|{}
\leq{}M\big(\rho|u|^2{}+{}\rho^\gamma{}+{}\rho|\etat_m(\rhob,0)|\big){}
\leq{} M\big(\eta^*(\rho,m){}+{}\rhob^{2\theta}\rho\big).
\]
It follows by the energy estimate \eqref{Energy_est} that
\[
\big|\int_a^b J_2\, \omega\, dr\big|{}\leq{} M(\supp\, \omega,T)\big(1{}+{}\rhob^\gamma b^n\big),
\]
for any nonnegative smooth function $\omega$
with $\supp \, \omega\subset (a, b)$.

We write
\[
\etat_r{}={}\rho_r\etat_\rho{}+{}m_r\etat_m{}={}\rho_r(\etat_\rho+u\etat_m){}+{}\rho\etat_mu_r.
\]
Then we consider the integral
\begin{eqnarray*}
\int_a^b J_3\,\omega\,dr
{}&=&{} \e\int_{Q_T}\rho\big(\etat_\rho{}+{}u\etat_m\big) \omega_r \, r^{n-1}drd\tau
{}-{}(n-1)\e\int_{Q_T}\rho\big(\etat_\rho{}+{}u\etat_m\big)\omega\,r^{n-2}drd\tau \notag \\
{}&&-{} \e\int_{Q_T} \rho \big(\rho_r(\etat_\rho+u\etat_m)_\rho{}+{}
u_r(\etat_\rho+u\etat_m)_u{}-{}\etat_mu_r\big)\omega\,r^{n-1}drd\tau.
\end{eqnarray*}

Noticing that $\etat_\rho{}+{}u\etat_m{}={}\etach_\rho+u\etach_m{}+{}const.$ and
using  Step 2 and estimates \eqref{Energy_est}--\eqref{est:energy_2}
and \eqref{3.8a},
we obtain
\begin{eqnarray*}
\big| \int_a^b J_3(t,r)\,\omega \,dr\big|
\leq M(a_1,T,\|\omega\|_{C^1})
{}+{}\frac{1}{2}\int_{Q_T}\big(\rho|u|^3{}+{}\rho^{\gamma+\theta}\big)\omega\,r^{n-1}drd\tau.
\end{eqnarray*}

To estimate $I_6$,
employing that $|(\etat_m)_\rho|\leq M\rho^{\theta-1}$,  $|(\etat_m)_u|{}\leq{} M$,
and the energy estimate \eqref{Energy_est}, we have
\[
|I_6|{}\leq{} M\frac{\delta^2}{\e}\int_0^T\int_r^b \rho^3\,r^{n-1}drd\tau{}
\leq{}M\frac{\delta^2}{\e^2}b^n\leq{}M\frac{\delta}{\e}b^n,
\]
where we have used the result of Lemma \ref{lemma 3.2}
and $\frac{\delta}{\e}<1$ for small $\e>0$  in the last inequality.

\medskip
The last term $I_7$ is estimated in the similar fashion:
\[
\big|\int_a^b I_7\,\omega \, dr\big|{}\leq{} M(\supp\, \omega)\frac{\delta^2}{\e}b^n
\leq{} M(\supp\, \omega)\frac{\delta}{\e}b^n,
\]
since $\delta<1$ for small $\e>0$.

Finally, we multiply equation \eqref{eq:nash-1} by the nonnegative smooth function $\omega$,
integrate it over $(a,b)$, and use estimate \eqref{id:0.1}, together with
the above estimates for $I_j, j=1,\cdots, 7$, and an appropriate choice of $\delta$  to obtain
\begin{eqnarray*}
&&\int_{Q_t}\big(\rho|u|^3+\rho^{\gamma+\theta}\big)\,\omega\,r^{n-1}drd\tau\\
&& \leq M\big(1+\rhob^\gamma b^n {}+{}\frac{\delta}{\e}b^n\big)
+ \frac{1}{2} \int_{Q_t}\big(\rho|u|^3{}+{}\rho^{\gamma+\theta}\big)\,\omega\,r^{n-1}drd\tau.
\end{eqnarray*}
This completes the proof.
\end{proof}
\end{subsection}

\smallskip
\begin{subsection}{Weak Entropy Dissipation Estimates}
Let $a=a(\e)\to0$ and $b=b(\e)\to \infty$.
We choose $\rhob{}={}\rhob(\e){}\to{}0$ and $\delta{}={}\delta(\e){}\to{}0$ such that
\begin{equation}
\label{size delta}
\rhob^\gamma b^n+
\frac{\delta}{\e}b^{n}   {}\leq{} M
 \qquad \mbox{uniformly in $\e$}.
\end{equation}
With this choice $(\rhob,\delta)$,
the estimates on the lemmas in \S 3.1 are uniform in $\e{}\to{}0$.

Given a sequence of the initial data functions as in Theorem \ref{main},
denote $(\rho^\e,m^\e)$ by the corresponding solution of the viscosity equations \eqref{eq:NS}
on $Q^\e{}={}[0,\infty)\times[a(\e),b(\e)]$ with $\rhob=\rhob(\e)$ as above.

\begin{proposition}
Let $(\eta,q)$ be an entropy pair of system \eqref{eq:Eu} with form \eqref{eta}--\eqref{q}
for a smooth, compactly supported function $\psi(s)$ on $\mathbb{R}$.
Then the entropy dissipation measures
\begin{equation}\label{dissipation-measures}
\eta(\rho^\e,m^\e)_t + q(\rho^\e,m^\e)_r
\qquad \mbox{are compact in }\,\, H_{loc}^{-1}.
\end{equation}
\end{proposition}

\begin{proof} We divide the proof into seven steps.

\medskip
1. Denote $\eta^\e{}={}\eta(\rho^\e,m^\e)$, $q^\e{}={}q(\rho^\e,m^\e)$, and $m^\e{}={}\rho^\e u^\e$.
We compute
\begin{eqnarray}
\label{approx_entropy}
\eta^\e_t{}+{}q^\e_r
{}&=&{}-\frac{n-1}{r}\rho u^\e\left(\eta^\e_{\rho} +u^\e\eta^\e_m \right)
{}+{} \e\frac{n-1}{r}\Big(\rho^\e_r\eta^\e_{\rho}{}+{}r\big(\frac{1}{r}m^\e\big)_r\eta^\e_m \Big) \notag \\[2mm]
{}&&-{}\e\big(\rho^\e_r(\eta^\e_\rho)_r{}+{} m^\e_r(\eta^\e_m)_r\big){}+{}\e\eta^\e_{rr}{}-{}
(\delta\rho^2)_r\eta^\e_m\notag \\[2mm]
{}&=&{}I^\e_1{}+\cdots +{}I^\e_5.
\end{eqnarray}

\smallskip
2. We notice that
\begin{equation}
|I^\e_1(t,r)|{}\leq{}M\rho^\e|u^\e|\big(1+(\rho^\e)^\theta\big)
{}\leq{}M\big(\rho^\e|u^\e|^2+\rho^\e{}+{}(\rho^\e)^\gamma\big),
\end{equation}
bounded in $L^1\left(0,T; L^1_{loc}(0,\infty)\right)$, independent of $\e$
(all of the functions are extended by $0$ outside $(a,b)$).

\medskip
3. Next,
\begin{eqnarray}
I^\e_2{}={}\e\frac{n-1}{r^2}\big(\eta^\e-m^\e\eta^\e_m\big) + \e\big(\frac{n-1}{r}\eta^\e \big)_r
=: I_{2a}^\e+I_{2b}^\e.
\end{eqnarray}

Since
\[
|\eta^\e{}-{}m^\e\eta^\e_m|{}\leq{}M \big(\rho^\e{}+{}\rho^\e|u^\e|^2\big),
\]
then
\begin{equation}\label{3.26a}
I_{2a}^\e \to 0  \qquad \mbox{in $L^1_{loc}(\R_+^2)$ as $\e\to 0$}.
\end{equation}
On the other hand, if $\omega$ is smooth and compactly supported on $\mathbb{R}^2_+$, then
\begin{eqnarray}
\e\left|\int_{Q^\e} I_{2b}^\e\omega(t,r)\,drdt\right|
&=&\e\left|\int\frac{n-1}{r}\eta^\e\omega_r\,drdt\right|\notag\\
&\leq& \e M(\supp\,\omega)\|\rho^\e\|_{L^{\gamma+1}(\supp\,\omega)}\|\omega\|_{H^1(\mathbb{R}^2_+)}.
\notag
\end{eqnarray}
Since $\|\rho^\e\|_{L^{\gamma+1}(\supp\,\omega)}$ is bounded, independent of $\e$ (see \eqref{HI_est_1}),
the above estimate shows that
\begin{equation}\label{3.26b}
I_{2b}^\e{}\to{}0\qquad  \mbox{in  $H^{-1}_{loc}(\mathbb{R}^2_+)$ as $\e\to 0$}.
\end{equation}

\medskip
4. For $I_3^\e$, we use Lemma \ref{2.1} to obtain
\begin{eqnarray}
|I_3^\e|&=&\e |\langle \nabla^2\eta(\rho^\e, m^\e)(\rho_r^\e,m^\e_r), (\rho_r^\e,m^\e_r)\rangle|\notag\\
&\le &  M_\psi\, \e \langle \nabla^2\bar{\eta}^*(\rho^\e, m^\e)(\rho_r^\e,m^\e_r), (\rho_r^\e,m^\e_r)\rangle.
\notag \label{I-3a}
\end{eqnarray}

Combining \eqref{I-3a} with Proposition \ref{2.1} and Lemma \ref{Energy_control_lemma}, we
conclude that
\begin{equation}\label{I-3b}
I_3^\e  \qquad \mbox{is uniformly bounded in $L^1(0,T; L^1_{loc}(0,\infty))$}.
\end{equation}

\medskip
5. To show that $I^\e_4{}\to{}0$ in $H^{-1}_{loc}$ as $\e\to 0$,
we need the following claim, adopting the arguments from \cite{LPS}.

{\it Claim}:
{\it Let $K\subset (0,\infty)$ be a compact subset. Then, for any $0<\Delta<1$ and $\e>0$,
\begin{equation}
\label{small_rho_est_1}
\int_0^T\int_K\e^{\frac{3}{2}}|\rho^\e_r|^2\,drdt{}\leq{}M\big(\sqrt{\e}\Delta^{\frac{\gamma}{2}}{}+{}\Delta+\e\big).
\end{equation}
In particular,
$$
\int_0^T\int_K\e^{\frac{3}{2}}|\rho^\e_r|^2\,drdt{}\to{}0,
$$
and
$$
\e\eta^\e_r{}\to{}0 \qquad\,\,\,\mbox{in}\,\, L^p(0,T;L^p_{loc}(0,\infty)) \quad\, \mbox{for}\,\,
p:=2-\frac{2}{\gamma+1}\in{}(1,2).
$$
}

Now we prove the claim.
For the simplicity of notation, we suppress superscript $\e$ in all of the functions.
Define
\[
\phi(\rho){}={}\left\{\begin{array}{ll}
\frac{\rho^2}{2}, & \rho<\Delta,\\[2mm]
\frac{\Delta^2}{2}{}+{}\Delta(\rho-\Delta), &\rho\geq\Delta,
\end{array}
\right.
\]
so that
\begin{eqnarray*}
&&\phi''(\rho){}={}\chi_{\{\rho<\Delta\}}(\rho),\\
&&\rho\phi'(\rho)-\phi(\rho){}={}\frac{\rho^2}{2} \qquad \mbox{for $\rho<\Delta$},\\
&&\rho\phi'(\rho)-\phi(\rho){}={}\frac{\Delta^2}{2} \qquad\mbox{for $\rho\geq\Delta$,}
\end{eqnarray*}
where $\chi_{A}(\rho)$ is the indicator function that is $1$ when $\rho\in A$ and $0$ otherwise.

\smallskip
Let $\omega(r)$ be a nonnegative smooth, compactly supported function on $(0,\infty)$.
We compute from the continuity equation, the first equation, in \eqref{eq:NS}:
\begin{eqnarray}
&& (\phi\omega)_t+(\phi u \omega)_r{}-{}\phi u\omega_r{}
-{}\frac{1}{2}\big(\rho^2\chi_{\{\rho<\Delta\}}{}+{}\delta^2\chi_{\{\rho>\Delta\}} \big)\omega u_r{}
+{}\frac{n-1}{r}\rho u\min\{\rho,\Delta\} \notag \\
&&=\e(\phi'\omega\rho_r)_r{}-{}\e\min\{\rho,\Delta\}\omega'\rho_r
{}+{}\frac{(n-1)\e}{r}\omega\min\{\rho,\Delta\}\rho_r{}
-{}\e\omega|\rho_r|^2\chi_{\{\rho<\Delta\}}. \label{3.32}
\end{eqnarray}
Integrating \eqref{3.32} over $(0,T)\times(0,\infty)$, we obtain
\begin{eqnarray}
&&\int_0^T\int\e\omega|\rho_r|^2\chi_{\{\rho<\Delta\}}\,drdt\notag\\
&&={}-\int \phi\omega\big|^T_0\,dr{}+{}\int_0^T\int\phi u\omega_r\,dr dt \notag \\
&&\quad +\frac{1}{2}\int_0^T\int \big(\rho^2\chi_{\{\rho<\Delta\}}{}+{}\delta\chi_{\{\rho>\Delta\}}\big)\omega u_r\,drdt{}
-{}\int_0^T\int \frac{n-1}{r}\rho u\min\{\rho,\Delta\}\,drdt \notag \\
&&\quad -\int_0^T\int\e\min\{\rho,\Delta\}\omega'\rho_r\,drdt{}
+{}\int_0^T\int\frac{(n-1)\e}{r}\omega\min\{\rho,\Delta\}\rho_r\,drdt\notag\\
&&={}J_1{}+ \cdots +{}J_6.
\end{eqnarray}

We estimate the integrals on the right:
\begin{eqnarray}
|J_1|{}\leq{}M(\supp\,\omega)\Big(\Delta^2+\Delta\int_0^T\int_{\supp\,\omega}\rho\,drdt\Big)
{}\leq M(\supp\,\omega,T)\Delta;
\end{eqnarray}
\begin{eqnarray}
|J_2|{}&\leq&{}\int_0^T\int_{\supp\,\omega}\big(\Delta|\rho u|\chi_{\{\rho<\Delta\}}{}
+{}(\Delta^2+\Delta\rho)|u|\chi_{\{\rho>\Delta\}}\big)\,drdt \notag \\
{}&\leq&{} \Delta\int_0^T\int_{\supp\,\omega}\big(\rho{}+{}\rho|u|^2\big)dtdt{}\notag\\
{}&\leq&{}M(\supp\,\omega,T)\Delta;
\end{eqnarray}
\begin{eqnarray}
|J_3|{}\leq{}\frac{\Delta^{\frac{3}{2}}}{\sqrt{\e}}\int_0^T\int_{\supp\,\omega} \big(\rho{}+{}\e\rho|u_r|^2\big)\,drdt
{}\leq{}M(\supp\,\omega,T)\frac{\Delta}{\sqrt{\e}};
\end{eqnarray}
\begin{eqnarray}
|J_4|{}\leq{}M(\supp\,\omega)\Delta\int_0^T\int_{\supp\,\omega}\big(\rho{}+{}\rho|u|^2\big)\,drdt
{}\leq M(\supp\,\omega,T)\Delta;
\end{eqnarray}
\begin{eqnarray}
|J_5|{}&\leq&{}\sqrt{\e} \Delta^{\frac{\gamma}{2}}\int_0^T\int_{\supp\,\omega}
\sqrt{\e}\rho^{\frac{\gamma-2}{2}}|\rho_r|\,drdt{}+{}\e\int_0^T\int_{\supp\,\omega}\rho |\rho_r|\chi_{\{\rho<\Delta\}}\omega'\,drdt \notag\\[2mm]
{}&\leq&{}
\frac{\e}{4}\int_0^T\rho^{\gamma-2}|\rho_r|^2\omega\,drdt{}+{}2\e\int_0^T\int_{\supp\,\omega}\rho^2\frac{|\omega'|^2}{\omega}\,drdt
{}+{}\sqrt{\e}\Delta^{\frac{\gamma}{2}}M(\supp\,\omega,T) \notag \\[2mm]
{}&\leq&{} \frac{\e}{4}\int_0^T\int|\rho_r|^2\omega\,drdt{}+{}\e M(\supp\,\omega,T) \notag\\[2mm]
  &&+ \sqrt{\e}\Delta^{\frac{\gamma}{2}}M(\supp\,\omega, T).
\end{eqnarray}
Moreover, $J_6$ is estimated in the same way as $J_5$.
Thus, estimate \eqref{small_rho_est_1} is proved.

Now we prove the second part of the claim.

\smallskip
Notice that
\begin{eqnarray}
|\eta_r|{}\leq{}M\big(|\rho_r||\eta_\rho+u\eta_m|{}+{}\rho |u_r|\big)
{}\leq{}M\big(|\rho_r|(1+\rho^\theta){}+{}\rho|u_r|\big).
\end{eqnarray}
Let $q\in (1,2)$ to be chosen later on. Compute
\begin{eqnarray}
\int_0^T\int_K \e^q|\eta_r|^q\,drdt
{}&\leq&{}M\int_0^T\int_K\e^q|\rho_r|^q\,drdt{}+{}
\int_0^T\int_K\e^q\big||\rho_r|\rho^\theta{}+{}\rho|u_r|\big|^q\,drdt \notag \\[2mm]
{}&\leq&{} \Delta{}+{}\frac{M}{\Delta}\int_0^T\int_K\e^{2q}|\rho_r|^2\,drdt \notag \\[2mm]
&&{}+{}M
\int_0^T\int_K\e^p\rho^{\frac{q}{2}}\big(|\rho^{\frac{\gamma-2}{2}}\rho_r|^q
{}+{}|\rho^{\frac{1}{2}}u_r|^q\big)\,drdt \notag \\[2mm]
{}&\leq&{}\Delta{}+{}\frac{M}{\Delta}\int_0^T\int_K\e^{\frac{3}{2}}|\rho_r|^2\,drdt \notag\\[2mm]
&&{}+{}\e^{q-1}M\int_0^T\int_K\big(\e(\rho^{\gamma-2}|\rho_r|^2{}+{}\rho|u_r|^2){}+{}\e\rho^{\frac{q}{2-q}}\big)\,drdt \notag \\[2mm]
{}&\leq&{}\Delta{}+{}\frac{M}{\Delta}\int_0^T\int_K\e^{\frac{3}{2}}|\rho_r|^2\,drdt{}+{}\e^{q-1}C(T,K),
\end{eqnarray}
provided that
$\frac{2}{2-q}={}\gamma+1$, which holds if and only if $q{}={}2 - \frac{2}{\gamma+1}$.
Combining this with estimate \eqref{small_rho_est_1}, we arrive at the conclusion of the claim.

\smallskip
6. Consider the last term $I_5^\e$. This term is bounded in $L^1(0,T\,:\, L^1_{loc}(0,\infty))$.
Indeed, for a compact set $K\subset(0,\infty)$,
using the energy estimates \eqref{Energy_est} and Lemma \ref{lemma 3.2}, we obtain
\begin{eqnarray*}
\int_0^T\int_K|I_5|\,drdt{}&\leq&{}M_\psi\int_0^T\int_K \delta\rho|\rho_r|\,drdt\\
&\leq& M(\psi,K)\Big(1 {}+{}\frac{\delta^2}{\e}\int_0^T\int_K \rho^{4-\gamma}\,drdt\Big)\\
&\leq& M(\psi,K)\Big(1{}+{}\frac{\delta^2}{\e} +\frac{\delta^2}{\e}\int_0^T\int_K\rho^3\,drdt\Big)\\
&\leq& M(\psi,K)\Big(1{}+{}\frac{\delta^2}{\e} +\frac{\delta^2}{\e}b^n\Big).
\end{eqnarray*}
From the choice of $\delta$, the term on the right is uniformly bounded in $\e$.

\smallskip
7. Combining Steps 1--6, we conclude
\begin{equation}
\eta(\rho^\e, m^\e)_t{}+{}q(\rho^\e, m^\e)_r{}={}f^\e{}+{}g^\e,
\end{equation}
where $f^\e$ is bounded in $L^1\left(0,T; L^1_{loc}(0,\infty)\right)$ and $g^\e{}\to{}0$ in
$W^{-1,q}_{loc}(\mathbb{R}^2_+)$ for some $q{}\in{}(1,2)$.
This implies that, for $1<q_1<2$,
\begin{equation}\label{h-3}
\eta(\rho^\e, m^\e)_t{}+{}q(\rho^\e, m^\e)_r{}
\quad\mbox{are confined in a compact subset of }\,\,
W^{-1,q_1}_{loc}.
\end{equation}

\smallskip
On the other hand, using formulas \eqref{eta}--\eqref{q}
and the estimates in Proposition \ref{2.1} and Lemma \ref{HI:2},
we obtain that, for any smooth, compactly supported function $\psi(s)$ on $\R$,
$$
\eta(\rho^\e,m^\e),\,q(\rho^\e,m^\e) \qquad \mbox{are
uniformly bounded in }  L^{q_2}_{loc}(\R_+^2),
$$
for $q_2=\gamma+1>2$ when $\gamma>1$.
This implies that, for some $q_2>2$,
\begin{equation}\label{h-4}
\eta(\rho^\e, m^\e)_t{}+{}q(\rho^\e, m^\e)_r{}
\qquad\mbox{are uniformly bounded in  }\,\, W^{-1,q_2}_{loc}.
\end{equation}

The interpolation compactness theorem ({\it cf.} \cite{Chen1,DCL})
indicates that, for $q_1>1$, $q_2\in(q_1, \infty]$, and $q_0\in [q_1,
q_2)$,
\[
\begin{array}{l}
\big(\mbox{compact set of}\: W^{-1,q_1}_{loc}(\R_+^2)\big)
\cap \big(\mbox{bounded set of}\: W^{-1,q_2}_{loc}(\R_+^2)\big)\\[1mm]
\subset \big(\mbox{compact set of}\: W^{-1,q_0}_{loc}(\R_+^2)\big),
\end{array}
\]
which is a generalization of Murat's lemma in \cite{Murat,Tartar}.
Combining this interpolation compactness theorem for $1<q_1<2,
q_2>2$, and $q_0=2$ with the facts in \eqref{h-3}--\eqref{h-4}, we
conclude the result.
\end{proof}
\end{subsection}

\begin{subsection}{Strong Convergence and the Entropy Inequality}

The {\it a priori} estimates and compactness properties we have obtained in \S 3.1--\S 3.2
imply that the viscous solutions satisfy the compensated compactness
framework in Chen-Perepelitsa \cite{CP}.
Then the compactness theorem established in \cite{CP} for the case $\gamma>1$  (also see LeFloch-Westdickenberg \cite{LW})
yields that
\[
(\rho^\e,m^\e){}\to{}(\rho,m) \qquad \mbox{a.e. $(t,r)\in\mathbb{R}^2_+\quad $ in $L^{p}_{loc}\left(\mathbb{R}^2_+\right)\times L^{q}_{loc}\left(\mathbb{R}^2_+\right)$}
\]
for $p\in [1,\gamma+1)$ and $q\in [1, \frac{3(\gamma+1)}{\gamma+3})$.
This requires the uniform bounds \eqref{HI_est_1}--\eqref{HI:2} and the estimate:
\[
|m|^q{}={}\rho^{\frac{q}{3}}|u|^q \rho^{\frac{2q}{3}}{}\leq{}\rho|u|^3{}+{}\rho^{\gamma+1}
\]
for $q{}={}\frac{3(\gamma+1)}{\gamma+3}$.

From the same estimates, we also obtain the convergence of the energy as $\e\to 0$:
$$
\eta^*(\rho^\e, m^\e)\to \eta^*(\rho, m)
\qquad \mbox{in $L^{1}_{loc}\left(\mathbb{R}^2_+\right)$}.
$$
Since the energy $\eta^*(\rho,m)$ is a convex function,
by passing to the limit in  \eqref{est:energy_2}, we obtain
\[
\int_{t_1}^{t_2} \int_0^\infty
\eta^*(\rho, m)
(t,r)\,r^{n-1}drdt{}\leq{}
(t_1-t_2)\int_0^\infty  \eta^*(\rho_0, m_0)
(r)\,r^{n-1}dr,
\]
which implies that, for {\it a.e.} $t\ge 0$,
\begin{equation}\label{energy-control-a}
\int_{\R_+}
\eta^*(\rho, m)
(t,r)\,r^{n-1}dr
{}\leq{}
\int_0^\infty  \eta^*(\rho_0, m_0)
(r)\,r^{n-1}dr.
\end{equation}
 This implies that there is no concentration formed in the density $\rho$ at the
origin $r=0$.

Furthermore, we multiply both sides of \eqref{eq:energy-2} by a smooth function $\varphi(t)\in C_0^1(\R_+)$ with $\varphi(0)=0$,
integrate it over $\R_+^2$, and pass to the limit $\e\to 0$ to obtain
$$
\int_{\R_+^2}\eta^*(\rho, m)\varphi'(t) \, dr dt \ge 0,
$$
which, together with \eqref{energy-control-a}, concludes \eqref{finite_energy}.

Finally, the energy estimates \eqref{Energy_est}--\eqref{est:energy_2} and the estimates
in Lemmas \ref{HI_est_1}--\ref{HI:2}
imply the equi-integrability of a
sequence of
$$
\eta_\psi^\e, \quad q_\psi^\e, \quad m^\e\partial_\rho\eta_{\psi}^\e,\quad
\frac{(m^\e)^2}{\rho^\e}\partial_m\eta_{\psi}^\e, \quad q_\psi^\e,
$$
for any $\psi(s)$ that is convex with subquadratic growth at infinity:
$\lim_{s\to\infty}\frac{|\psi(s)|}{s^2}=0$.

Passing to the limit in \eqref{approx_entropy} multiplied by $r^n$ and integrated against
a smooth compactly function supported on $(0,\infty)\times(0,\infty)$, we obtain \eqref{entropy_sol}.
\end{subsection}

\begin{subsection}{Limit in the Equations}

Let $\varphi(t,r)$ be a smooth, compactly supported function on $[0,\infty)\times[0,b(\e))$,
with $\varphi_r(t,r){}={}0$ for all $r$ close to $0$.
Assume that the viscosity solutions $(\rho^\e,$ $m^\e)$ are extended by $0$ outside of $[a(\e),b(\e)]$.
Multiplying the first equation
in \eqref{eq:NS} by $r^{n-1}\varphi$ and then integrating it over $\mathbb{R}^2_+$, we have
\begin{eqnarray*}
&&
\int_{\mathbb{R}^2_+} \Big(\rho^\e\varphi_t{}+{}m^\e\varphi_r{}+{}\e\rho^\e(\varphi_{rr}
  {}+{}\frac{n-1}{r}\varphi_r)\Big)\,r^{n-1}drdt{}\notag\\
&&\,\,\, +{}\int_{\mathbb{R}_+}\rho^\e_0(r)\varphi(0,r)\,r^{n-1}dr{}={}0.
\label{weak_form_eq_1}
\end{eqnarray*}
Note that, by the energy inequality, $\int_0^1(\rho^\e)^\gamma\,r^{n-1}dr$
is bounded, independent of $\e$, which implies that there is no concentration of mass at $r{}={}0.$

\smallskip
Passing to the limit in the above equation, we deduce
\begin{eqnarray*}
\int_{\mathbb{R}^2_+} \big(\rho \varphi_t{}+{}m\varphi_r\big){}\,r^{n-1}drdt
{}+{}\int_{\mathbb{R}_+}\rho_0(r)\varphi(0,r)\,r^{n-1}dr{}={}0,
\end{eqnarray*}
which can be extended to hold for all smooth, compactly supported function $\varphi(t,r)$
on $[0,\infty)\times[0,\infty)$, with $\varphi_r(t,0){}={}0$.

\smallskip
Consider now the momentum equation in \eqref{eq:Eu}. Let $\varphi(t,r)$ be
a smooth, compactly supported function on $[0,\infty)\times(a(\e),b(\e))$.
Multiplying the first equation
in \eqref{eq:NS} and then integrating it over $\mathbb{R}^2_+$, we obtain
\begin{eqnarray*}
&&\int_{\mathbb{R}^2_+} \Big(m^\e\varphi_t{}+{}\frac{(m^\e)^2}{\rho^\e}\varphi_r
{}+{}p_\delta(\rho^\e)\big(\varphi_r{}+{}\frac{n-1}{r}\varphi\big){}+{}\e m^\e\varphi_{rr}\Big)\,r^{n-1}drdt{}\\
&& +{}\int_{\mathbb{R}_+} m^\e_0(r)\varphi(0,r)\,r^{n-1}dr{}={}0.
\end{eqnarray*}

Passing to the limit, we find
\begin{eqnarray*}
\label{weak_form_eq_2}
\int_{\mathbb{R}^2_+} \Big(m\varphi_t{}+{}\frac{m^2}{\rho}\varphi_r{}+{}p(\rho)\big(\varphi_r{}+{}\frac{n-1}{r}\varphi\big)\Big)\,r^{n-1}drdt
{}+{}\int_{\mathbb{R}_+} m_0(r)\varphi(0,r)\,r^{n-1}dr{}={}0.
\end{eqnarray*}

Note that the term containing $\delta \rho^2$ converges to zero by Lemma \ref{lemma 3.2}
since $\delta=\delta(\e)\to 0$ as $\e\to 0$.

This equation can be extended for all smooth compactly supported function $\varphi(t,r)$
on $[0,\infty)\times[0,\infty)$ with $\varphi(t,0){}={}\varphi_r(t,0){}={}0$,
since $(\frac{m^2}{\rho}+\rho^\gamma)(t,r)r^{n-1}\in L^1_{loc}([0,\infty)\times [0,\infty))$.
\end{subsection}
\end{section}

\vspace{.25in}
\noindent
{\bf Acknowledgements:} The research of
Gui-Qiang G. Chen was supported in part by
the UK EPSRC Science and Innovation
Award to the Oxford Centre for Nonlinear PDE (EP/E035027/1),
the UK EPSRC Award to the EPSRC Centre for Doctoral Training
in PDEs (EP/L015811/1),
the NSFC under a joint project Grant 10728101, and
the Royal Society--Wolfson Research Merit Award (UK).
The research of Mikhail Perepelitsa was
supported in part by the NSF Grant DMS-1108048.
The authors would like to thank the Isaac Newton Institute for Mathematical Sciences,
Cambridge, for support and hospitality during the 2014 Programme on
{\it Free Boundary Problems and Related Topics} where work on this paper was undertaken.

\end{document}